\documentclass[a4paper]{article}

\usepackage{amsfonts}
\usepackage{graphicx}
\usepackage{amsthm}
\usepackage{amsmath}
\usepackage{amssymb}
\usepackage{enumerate}
\allowdisplaybreaks

\newtheorem{proposition}{Proposition}[section]
\newtheorem{lemma}[proposition]{Lemma}
\newtheorem{corollary}[proposition]{Corollary}
\newtheorem{theorem}[proposition]{Theorem}
\theoremstyle{definition}
\newtheorem{definition}[proposition]{Definition}

\newtheorem{remark}[proposition]{Remark}
\numberwithin{equation}{section}
\newtheorem*{example*}{Example}

\begin{document}

\begin{center}
\LARGE
\textbf{A sum-bracket theorem for simple Lie algebras}
\bigskip\bigskip

\large
Daniele Dona
\bigskip

\normalsize
Einstein Institute of Mathematics, The Hebrew University of Jerusalem

Edmond J. Safra Campus Givat Ram, Jerusalem 9190401, Israel

\texttt{daniele.dona@mail.huji.ac.il}
\bigskip\medskip
\end{center}

\begin{minipage}{110mm}
\small
\textbf{Abstract.} Let $\mathfrak{g}$ be an algebra over $K$ with a bilinear operation $[\cdot,\cdot]:\mathfrak{g}\times\mathfrak{g}\rightarrow\mathfrak{g}$ not necessarily associative. For $A\subseteq\mathfrak{g}$, let $A^{k}$ be the set of elements of $\mathfrak{g}$ written combining $k$ elements of $A$ via $+$ and $[\cdot,\cdot]$.

We show a ``sum-bracket theorem'' for simple Lie algebras over $K$ of the form $\mathfrak{g}=\mathfrak{sl}_{n},\mathfrak{so}_{n},\mathfrak{sp}_{2n},\mathfrak{e}_{6},\mathfrak{e}_{7},\mathfrak{e}_{8},\mathfrak{f}_{4},\mathfrak{g}_{2}$: if $\mathrm{char}(K)$ is not too small, we have growth of the form $|A^{k}|\geq|A|^{1+\varepsilon}$ for all generating symmetric sets $A$ away from subfields of $K$. Over $\mathbb{F}_{p}$ in particular, we have a diameter bound matching the best analogous bounds for groups of Lie type \cite{BDH21}.

As an independent intermediate result, we prove also an estimate of the form $|A\cap V|\leq|A^{k}|^{\dim(V)/\dim(\mathfrak{g})}$ for linear affine subspaces $V$ of $\mathfrak{g}$. This estimate is valid for all simple algebras, and $k$ is especially small for a large class of them including associative, Lie, and Mal'cev algebras, and Lie superalgebras.
\medskip

\medskip

\end{minipage}
\medskip

\section{Introduction}

The concept of \textit{growth} of a set inside an algebraic structure has been extensively investigated. Given a set $X$ inside some structure $G$ and an integer $k\geq 1$, how large is the set of elements of $G$ that can be written using the operations of $G$ and at most $k$ elements of $X$? Cases that have been studied include:
\begin{itemize}
\item abelian groups, starting with results on $\mathbb{Z}$ by Freiman (see the monograph \cite{Fre73}) and Ruzsa \cite{Ruz94}, which led to the development of \textit{sumset theory} and of results for arbitrary abelian groups as in \cite{GR07};
\item fields, starting with Erd\H{o}s and Szemer\'edi \cite{ES83}, who proved for sets of integers the first \textit{sum-product theorem} (see \S\ref{se:prelfields} for a few examples);
\item finite simple groups, chiefly with \textit{Babai's conjecture} \cite[Conj.~1.7]{BS88}, with many partial results dependent on the Classification of Finite Simple Groups (CFSG), such as \cite{BGT11} \cite{HS14} \cite{PS16} \cite{BDH21};
\item general non-abelian groups, for instance with the \textit{Helfgott-Lindenstrauss conjecture} \cite[\S 1.1]{GH14}, solved for all groups \cite[Thm.~1.6]{BGT12} in a weak quantitative sense, and more strongly in the recent preprint \cite{EMPS21} for $\mathrm{GL}_{n}(K)$ with bounded $n$;
\item rings, where sum-product theorems also apply to a certain extent \cite{Tao09}.
\end{itemize}
In all the cases above, the lesson is the same: for any set $X$, either $X$ grows quickly, or there are constraints on its structure.

In the present paper we consider the problem of growth in yet another class of objects, namely \textit{non-associative algebras}. Having two operations at our disposal, we would expect to have growth as in the case of sum-product theorems for fields and rings.
It turns out that this expectation is correct for simple Lie algebras $\mathfrak{g}$ of the form $\mathfrak{sl}_{n},\mathfrak{so}_{n},\mathfrak{sp}_{2n},\mathfrak{e}_{6},\mathfrak{e}_{7},\mathfrak{e}_{8},\mathfrak{f}_{4},\mathfrak{g}_{2}$ (\textit{classical Lie algebras}, see \S\ref{se:generic} for the choice of nomenclature), over a field $K$ with the mild restriction that $\mathrm{char}(K)$ is either $0$ or $\geq 3\dim(\mathfrak{g})$: in this case, the sets do indeed have exponential growth depending only on $\dim(\mathfrak{g})$, and the only obstacles lie in the structure of $K$. We call our result \textit{sum-bracket theorem}, in analogy with sum-product theorems.

Furthermore, we provide as an intermediate result a \textit{dimensional estimate}, which holds not only for Lie algebras, but for all finite-dimensional simple algebras $\mathfrak{g}$. In brief, for any $d$-dimensional subspace $V$ of $\mathfrak{g}$, a set $A$ has at most roughly $|A^{k}|^{\frac{d}{\dim(\mathfrak{g})}}$ elements lying on $V$, for some power $A^{k}$ of $A$. Moreover, $k$ is particularly small for a large class of algebras that encompasses \textit{associative algebras}, \textit{Lie algebras}, \textit{Mal'cev algebras}, and \textit{Lie superalgebras} as particular cases.

\subsection{Results}

We work often with general (finite-dimensional) algebras $\mathfrak{g}$: $(\mathfrak{g},+,[\cdot,\cdot])$ is an \textit{algebra} (over $K$) if $(\mathfrak{g},+)$ is a $K$-vector space and $[\cdot,\cdot]:\mathfrak{g}\times\mathfrak{g}\rightarrow\mathfrak{g}$ is bilinear. We do not require the algebra $\mathfrak{g}$ to be associative. For any $X,Y\subseteq\mathfrak{g}$, we write
\begin{align*}
X+Y& =\{x+y|x\in X,y\in Y\}, & [X,Y] & =\{[x,y]|x\in X,y\in Y\}.
\end{align*}
We also define recursively
\begin{align*}
X^{0} & =\{0\}, & X^{1} & =X, & X^{k} & =\bigcup_{0<j<k}\left((X^{j}+X^{k-j})\cup[X^{j},X^{k-j}]\right).
\end{align*}
A set $X\subseteq\mathfrak{g}$ is \textit{symmetric} if it is finite, $0\in X$, and $X=-X$. If $X$ is symmetric, the \textit{algebra generated by $X$} is $\langle X\rangle=\bigcup_{j=0}^{\infty}X^{j}$, whereas the space \textit{spanned} by $X$ is $\mathrm{span}_{K}(X)=\sum_{x\in X}Kx$.

We can now state the main result of the paper.

\begin{theorem}[Sum-bracket theorem]\label{th:main}
Let $\mathfrak{g}$ be a classical Lie algebra over a field $K$, and let $A\subseteq\mathfrak{g}$ be a symmetric set with $\mathrm{span}_{K}(\langle A\rangle)=\mathfrak{g}$.

Then, for some absolute constant $\varepsilon>0$ and some $k=e^{O(\dim(\mathfrak{g})^{2}\log\dim(\mathfrak{g}))}$, the following statements hold.
\begin{enumerate}[(i)]
\item\label{th:mainzero} If $\mathrm{char}(K)=0$, then
\begin{equation*}
|A^{k}|\geq|A|^{1+\varepsilon}.
\end{equation*}
\item\label{th:mainprime} If $\mathrm{char}(K)=p>0$, then
\begin{equation*}
|A^{k}|\geq\min\{|A|^{1+\varepsilon},p^{\dim(\mathfrak{g})}\}.
\end{equation*}
\item\label{th:mainstuck} Let $0<\delta<1$. If $\mathrm{char}(K)\geq 3\dim(\mathfrak{g})$ and there are no subfields $K'\leq K$ with $|A|^{\frac{1-\delta}{\dim(\mathfrak{g})}}\leq|K'|\leq|A|^{\frac{1+\delta}{\dim(\mathfrak{g})}}$, then
\begin{equation*}
|A^{k^{\lceil 1/\delta\rceil}}|\geq|A|^{1+\delta\varepsilon}.
\end{equation*}
\item\label{th:mainbig} Let $m\geq 1$ be an integer. If $K=\mathbb{F}_{q}$ with $\mathrm{char}(K)\geq 3\dim(\mathfrak{g})$ and $|A|\geq|\mathfrak{g}|^{\frac{1}{2}\left(1+\frac{1}{m}\right)}$, then
\begin{equation*}
A^{k^{m}}=\mathfrak{g}.
\end{equation*}
\end{enumerate}
\end{theorem}

The value of $k$ above is in line with the bound in \cite[Thm.~6.4]{BDH21} for classical Chevalley groups: up to standard manipulations, that result asserts that $|A^{k}|\geq|A|^{1+\varepsilon}$ for some absolute constant $\varepsilon>0$ and $k=e^{O(r^{4}\log r)}$, where $r$ is the rank of $G$.

In analogy with group diameters, for an algebra $\mathfrak{g}$ define the \textit{diameter} $\mathrm{diam}(\mathfrak{g})$ to be the smallest $k$ such that $A^{k}=\mathfrak{g}$ for all symmetric sets $A$ generating $\mathfrak{g}$. Theorem~\ref{th:main}\eqref{th:mainprime} implies the following.

\begin{corollary}[Diameter bound]\label{co:diam}
Let $p$ be a prime, and let $\mathfrak{g}$ be a classical Lie algebra over $\mathbb{F}_{p}$. Then
\begin{equation*}
\mathrm{diam}(\mathfrak{g})\leq(\log|\mathfrak{g}|)^{C}
\end{equation*}
for $C=O(\dim(\mathfrak{g})^{2}\log\dim(\mathfrak{g}))$.
\end{corollary}

Again, the bound in Corollary~\ref{co:diam} is qualitatively as sharp as the diameter bound in \cite[Thm.~1.2]{BDH21} for classical Chevalley groups, for which we have $\mathrm{diam}(G)\leq(\log|G|)^{O(r^{4}\log r)}$. We remark that this is not a hard limit, and that eliminating at least the logarithmic factor inside the exponent $C$ should be entirely possible (see Appendix~\ref{se:appfurther} for further comments).

Before stating the dimensional estimate result, we need some more definitions. If $\mathfrak{g}$ is a finite-dimensional vector space over a field $K$, a \textit{linear affine space} inside $\mathfrak{g}$ is the set of zeros in $\mathfrak{g}$ of a finite collection $\{f_{i}\}_{i}$ of polynomials of degree $1$ over $K$ (not necessarily homogeneous). Equivalently, $V$ is a linear affine space if, for every $x,x'\in V$ and every $k\in K$, the element $kx+(1-k)x'$ is also in $V$. Every linear affine space $L$ is a translate $W+x$ of a unique $K$-vector subspace $W$ of $\mathfrak{g}$, and $x$ can be taken to be any element of $V$ itself; the \textit{dimension} of $L$ is defined to be $\dim(V)=\dim(W)$.

When studying growth, we may sometimes want to restrict to using only $[\cdot,\cdot]$. Let $X,Y\subseteq\mathfrak{g}$. We define recursively the following pieces of notation:
\begin{align*}
X^{[0]}& =\mathcal{T}_{0}(X)=\mathcal{T}_{0}(X,Y)=\mathcal{S}_{0}(X,Y)=\{0\}, \\
X^{[1]}& =\mathcal{T}_{1}(X)=\mathcal{T}_{1}(X,Y)=\mathcal{S}_{1}(X,Y)=X, \\
X^{[k]}& =\bigcup_{0<j<k}[X^{[j]},X^{[k-j]}], \\
\mathcal{T}_{k}(X)& =[\mathcal{T}_{k-1}(X),X]\cup[X,\mathcal{T}_{k-1}(X)], \\
\mathcal{T}_{k}(X,Y)& =[\mathcal{T}_{k-1}(X,Y),Y]\cup[Y,\mathcal{T}_{k-1}(X,Y)]\cup[\mathcal{T}_{k-1}(Y),X]\cup[X,\mathcal{T}_{k-1}(Y)], \\
\mathcal{S}_{k}(X,Y)& =\bigcup_{0<j<k}\left([\mathcal{S}_{j}(X,Y),Y^{[k-j]}]\cup[Y^{[j]},\mathcal{S}_{k-j}(X,Y)]\right).
\end{align*}
Informally speaking: $X^{[k]}$ is the set of expressions constructed using only $[\cdot,\cdot]$; $\mathcal{T}_{k}(X)$ is the subset of $X^{[k]}$ in which, for any two $[\cdot,\cdot]$ in the construction, one is nested into the other; and $\mathcal{T}_{k}(X,Y)$ (respectively $\mathcal{S}_{k}(X,Y)$) is the set of expressions constructed taking an element of $\mathcal{T}_{k}(Y)$ (respectively $Y^{[k]}$) and replacing with an $X$ one of the occurrences of $Y$. We shall informally refer to elements of $\mathcal{T}_{k}(X)$ and $\mathcal{T}_{k}(X,Y)$ as \textit{towers}. We frequently adopt some shorthand notation, such as $\mathcal{T}_{k}(x,Y)=\mathcal{T}_{k}(\{x\},Y)$ and $\mathcal{T}_{\leq k}(X,Y)=\bigcup_{0\leq j\leq k}\mathcal{T}_{j}(X,Y)$, and similarly for the other definitions above.

We frequently consider algebras having at least one of the following two properties: the first is well-known, the second is defined here for the first time.

\begin{definition}\label{de:structure}
Let $\mathfrak{g}$ be an algebra over a field $K$. Then
\begin{itemize}
\item $\mathfrak{g}$ is \textit{simple} if $[V,\mathfrak{g}]\cup[\mathfrak{g},V]\not\subseteq V$ for all $K$-vector subspaces $\{0\}\subsetneq V\subsetneq\mathfrak{g}$, and
\item $\mathfrak{g}$ is a \textit{tower algebra} if for all $x\in\mathfrak{g}$, all $Y\subseteq\mathfrak{g}$ symmetric, and all $k\geq 1$ we have
\begin{equation*}
\mathcal{S}_{\leq k}(x,Y)\subseteq\mathrm{span}_{K}(\mathcal{T}_{\leq k}(x,Y)).
\end{equation*}
\end{itemize}
\end{definition}

Again informally speaking, $\mathfrak{g}$ is a ``tower algebra'' when it is only moderately non-associative: an expression of $\mathfrak{g}$ with non-nested $[\cdot,\cdot]$ can still be rearranged so as to form towers.

We are now ready for the dimensional estimate result.

\begin{theorem}[Dimensional estimate]\label{th:main0}
Let $\mathfrak{g}$ be a finite-dimensional simple algebra over a field $K$, not necessarily associative, and let $A\subseteq\mathfrak{g}$ be a symmetric set with $\mathrm{span}_{K}(\langle A\rangle)=\mathfrak{g}$. Let $V\subseteq\mathfrak{g}$ be a linear affine space of $\mathfrak{g}$. Then
\begin{equation*}
|A\cap V|\leq|A^{k}|^{\frac{\dim(V)}{\dim(\mathfrak{g})}}
\end{equation*}
for some $k=e^{O(\dim(\mathfrak{g}))}$. If $\mathfrak{g}$ is also a tower algebra, we can take $k=O(\dim(\mathfrak{g})^{2})$.
\end{theorem}

The bound above is considerably tighter than the corresponding one in the case of Chevalley groups, such as in \cite[Thm.~A.7]{BDH21}.

We shall show in \S\ref{se:prelturr} that several noteworthy algebras are tower algebras.

\begin{corollary}\label{co:main02}
Theorem~\ref{th:main0} holds with $k=O(\dim(\mathfrak{g})^{2})$ for finite-dimensional
\begin{enumerate}[(i)]
\item\label{co:main02assoc} simple associative algebras, with $[\cdot,\cdot]$ the multiplication,
\item\label{co:main02lie} simple Lie algebras, with $[\cdot,\cdot]$ the Lie bracket,
\item\label{co:main02liesuper} simple Lie superalgebras, with $[\cdot,\cdot]$ the Lie superbracket, and
\item\label{co:main02malcev} simple Mal'cev algebras with $\mathrm{char}(K)\neq 2$, with $[\cdot,\cdot]$ the Mal'cev bracket.
\end{enumerate}
\end{corollary}

\subsection{Overview of the argument}\label{se:introover}

In order to prove Theorem~\ref{th:main}, we go first through the dimensional estimate of Theorem~\ref{th:main0}: this will sound familiar to the reader that has experience with growth in linear algebraic groups. The main step (see the proof of Theorem~\ref{th:dimest}) is an induction procedure in which maps $f$ are applied to the space $V$ in order to increase or decrease its dimension, until we reach $\dim(V)\in\{0,\dim(\mathfrak{g})\}$.

The capital advantage in the case of algebras with respect to the case of groups is that, since $[\cdot,\cdot]$ is bilinear, all objects involved are linear spaces and all maps are linear. To prove the dimensional estimate for linear algebraic groups, for general maps $f:X\rightarrow Y$ we need to control the contribution of the fibres of $f$ that have dimension larger than $\dim(X)-\dim(f(X))$: in our case however, all fibres have the same dimension and linearity is preserved throughout (Proposition~\ref{pr:linearmap}). Another advantage comes from the routine called \textit{escape from subvarieties}, in which we pass to a small power $A^{k}$ of $A$ to be sure that we get out of the variety of ``special'' (i.e.\ especially bad) elements: however, escaping from linear spaces is much more rapid when $\mathfrak{g}$ is simple and/or a tower algebra (\S\ref{se:esc}).

Now we pass to Theorem~\ref{th:main}. The key idea is the following. Applying dimensional estimates (which are upper bounds) to linear spaces of the form $f^{-1}(x)$, we get analogous lower bounds for the image of $f$. By choosing $f$ appropriately, we obtain a $1$-dimensional subspace $V$ having roughly at least $|A|^{\frac{1}{\dim(\mathfrak{g})}}$ elements of $A$ lying on it: this process appears in \S\ref{se:descent}. $V$ is just a copy of the field $K$, so we apply the sum-product theorem: the sum in $K$ corresponds to the sum in $\mathfrak{g}$, and the product in $K$ is achieved through $[\cdot,\cdot]$ by bilinearity. Thus, we can produce $|A|^{\frac{1+\varepsilon}{\dim(\mathfrak{g})}}$ elements in $V$: after we stick $V$ in linearly independent directions, we get $|A|^{1+\varepsilon}$ elements overall by direct sum: this step is contained in \S\ref{se:main}.

The argument above is quite old. The first result proving Babai's conjecture for a family of groups, due to Helfgott \cite{Hel08}, used the same technique: passing from a set $A\subseteq\mathrm{SL}_{2}(\mathbb{F}_{p})$ to a set of traces, of size roughly $|A|^{1/3}$, then using sum-product theorems in this new set and finally going back to $\mathrm{SL}_{2}(\mathbb{F}_{p})$.

Unlike for dimensional estimates, for this second part we focus on classical Lie algebras. These objects have a desirable additional property: up to few manipulations, every subspace $V$ of $\mathfrak{g}$ we encounter is such that, for some $x\in V$, we have $[x,x]=0$ (by definition of Lie bracket) but $[y,x]\neq 0$ for some other $y\in V$. Thus we can pass from $V$ to a new space $[V,x]$ and decrease its dimension without making it go to $0$: the resulting controlled descent in dimensions lets us safely reach the final $1$-dimensional subspace. This additional property is shown to hold for classical Lie algebras in \S\ref{se:generic}, with one technical detail postponed to Appendix~\ref{se:appextr}.

We reserve Appendix~\ref{se:appfurther} to comment on possible further directions of study.

\section{Preliminaries}\label{se:prel}

In this section, we collect the initial results that we need for the main argument: elementary facts about sets and their span (\S\ref{se:prelspan}), tower algebras (\S\ref{se:prelturr}), and sum-product theorems for fields (\S\ref{se:prelfields}).

\subsection{Size growth and span growth}\label{se:prelspan}

We begin with a few elementary observations about sets in algebras. Trivially, if $X^{k}=X^{k+1}=\ldots=X^{2k}$ then $X^{k}=\langle X\rangle$. Aside from the algebra $\langle X\rangle$ and the vector space $\mathrm{span}_{K}(X)$, the set $X$ yields also the abelian group $\mathrm{span}_{+}(X)=\sum_{x\in X}\mathbb{Z}x$: clearly $\mathrm{span}_{+}(X)\subseteq\mathrm{span}_{K}(X)\cap\langle X\rangle$, but in general neither of the sets $\mathrm{span}_{K}(X),\langle X\rangle$ is contained in the other (although a small power of $X$ does span the same space as $\langle X\rangle$: see Proposition~\ref{pr:span}). By definition $X^{[k]}\subseteq X^{k}$, but by the bilinearity of $[\cdot,\cdot]$ we have also
\begin{equation}\label{eq:braspan}
X^{k}\subseteq\mathrm{span}_{+}(X^{[\leq k]}).
\end{equation}

It may not be true, unlike in the case of groups, that $X^{k}=X^{k+1}$ implies $X^{k}=\langle X\rangle$; in particular we may not have in general that $|X^{k}|\geq k$. As a matter of fact, if we had just $|X^{2k}|=|X^{k}|+1$ then the size of $X^{k}$ would only be logarithmic in $k$. This does not happen, fortunately: the size of $X^{k}$ is still at least $k^{\varepsilon}$ for some $\varepsilon>0$, and this will be enough for us.

\begin{lemma}\label{le:linsize}
Let $X$ be a symmetric set inside an algebra, not necessarily associative. Then there is some absolute constant $\varepsilon>0$ such that $|X^{k}|\geq\min\{k^{\varepsilon},|\langle X\rangle|\}$ for all $k\geq 1$.
\end{lemma}

\begin{proof}
The function $f(k)=|X^{k}|$ is increasing, because $0\in X$ implies $X^{k}\subseteq X^{k+1}$, and its supremum is $|\langle X\rangle|$. If $|X|=1$ then $X=\{0\}=X^{k}=\langle X\rangle$ and there is nothing to prove; in addition, for $|X|\geq 2$ the statement is trivially true for all $k$ lower than a given constant. It is then sufficient to show that for any $k\geq 1$ either $X^{4k}=\langle X\rangle$ or $|X^{6k}|\geq\frac{3}{2}|X^{k}|$.

First of all $X^{k}$ is a subset of $\langle X\rangle$, which forms a group with the addition, and $0\in X^{k}$. By \cite[Thm.~1]{Ols84} with $A=B=X^{k}$, either $X^{k}+X^{k}+X^{k}=X^{k}+X^{k}$ or $|X^{k}+X^{k}|\geq\frac{3}{2}|X^{k}|$. If the latter holds, $|X^{6k}|\geq|X^{2k}|\geq\frac{3}{2}|X^{k}|$ as well. Suppose the former holds instead: in particular we have $\mathrm{span}_{+}(X^{k})=X^{k}+X^{k}$.

If there is no element $x\in X^{4k}\setminus X^{2k}$, then trivially $X^{4k}=\langle X\rangle$. Assume that there is one such $x$; then, for all $x_{1},x_{2}\in X^{k}+X^{k}$ we have
\begin{equation*}
x+x_{1}=x_{2} \ \ \Longrightarrow \ \ x=x_{2}-x_{1}\in\mathrm{span}_{+}(X^{k})=X^{k}+X^{k}\subseteq X^{2k}
\end{equation*}
which contradicts the assumption. Hence, $X^{k}+X^{k}$ is disjoint from $x+X^{k}+X^{k}$, and $|X^{6k}|\geq 2|X^{k}+X^{k}|\geq 2|X^{k}|\geq\frac{3}{2}|X^{k}|$.
\end{proof}

We improve also in another sense on the trivial observation at the beginning of the section: not only does the size grow between $k$ and $2k$, but also the span. The following lemma is folklore, and appears for instance also in \cite[Prop.~2.3]{GK20}.

\begin{lemma}\label{le:logspan}
Let $X$ be a symmetric set inside an algebra over a field $K$, not necessarily associative. Define $W_{k}:=\mathrm{span}_{K}(X^{k})$.

Then, if $W_{k}=W_{k+1}=\ldots=W_{2k}$, we have $W_{k'}=W_{k}$ for all $k'\geq k$.
\end{lemma}

\begin{proof}
We just need to show that $W_{2k}=W_{2k+1}=W_{2k+2}$. By \eqref{eq:braspan} we have also $W_{k}=\mathrm{span}_{K}(X^{[\leq k]})$, so $W_{2k+1}$ is spanned by the union of $W_{2k}$ and of $[W_{j},W_{2k+1-j}]$ for all $0<j<2k+1$. For $0<j\leq k$ we have $W_{2k+1-j}=W_{k}$ by the hypothesis, and for $k<j<2k+1$ we have $W_{j}=W_{k}$ and $2k+1-j\leq k$. Hence, $W_{2k+1}$ is contained in the span of the union of $W_{2k}$ and of $[W_{j},W_{k}]$ for all $0<j\leq k$, which is $W_{2k}$ itself. We argue similarly for $W_{2k+2}$.
\end{proof}

\subsection{Tower algebras}\label{se:prelturr}

In this subsection, we give some concrete examples of tower algebras.

\begin{definition}\label{de:liemalcev}
A \textit{Lie algebra} $\mathfrak{g}$ over $K$ is a $K$-vector space endowed with a bilinear operation $[\cdot,\cdot]$ (the \textit{Lie bracket}) satisfying
\begin{enumerate}
\item $[x,x]=0$ for all $x\in\mathfrak{g}$, and
\item $[[x,y],z]+[[y,z],x]+[[z,x],y]=0$ for all $x,y,z\in\mathfrak{g}$ (Jacobi identity).
\end{enumerate}
A \textit{Mal'cev algebra} $\mathfrak{g}$ over $K$ is a $K$-vector space endowed with a bilinear operation $[\cdot,\cdot]$ (the \textit{Mal'cev bracket}) satisfying
\begin{enumerate}
\item $[x,y]=-[y,x]$ for all $x,y\in\mathfrak{g}$ (anti-commutativity), and
\item $[[x,y],[x,z]]=[[[x,y],z],x]+[[[y,z],x],x]+[[[z,x],x],y]$ for all $x,y,z\in\mathfrak{g}$.
\end{enumerate}
\end{definition}

All Lie algebras are Mal'cev algebras, and all simple Lie algebras are simple Mal'cev algebras. Conversely, if $\mathrm{char}(K)\notin\{2,3\}$ a simple Mal'cev algebra is either a simple Lie algebra or a $7$-dimensional algebra of a known type (see \cite[Thm.~11]{Kuz68}).

Assuming Theorem~\ref{th:main0}, the various cases of Corollary~\ref{co:main02} follow from Proposition~\ref{pr:liemalcev}, Remark~\ref{re:associative}, and Remark~\ref{re:liesuper}.

\begin{proposition}\label{pr:liemalcev}
Lie algebras are tower algebras. Mal'cev algebras are tower algebras when $\mathrm{char}(K)\neq 2$.
\end{proposition}

\begin{proof}
When $\mathrm{char}(K)\neq 2$ a Mal'cev algebra $\mathfrak{g}$ satisfies
\begin{equation}\label{eq:malcev}
[[[a,b],c],d]+[[[b,c],d],a]+[[[c,d],a],b]+[[[d,a],b],c]+[[a,c],[d,b]]=0
\end{equation}
for all $a,b,c,d\in\mathfrak{g}$ (see \cite[Prop.~2.21]{Sag61}). The proof is largely similar for Lie and Mal'cev algebras, so we treat the two cases together.

We proceed by induction on $k$. For $k\leq 3$ the statement is checked directly by definition, so assume that $k\geq 4$ and that the statement is true for all $k'<k$. Let $z\in\mathcal{S}_{k}(x,Y)$. In the cases
\begin{align*}
z & =[[z_{1},y_{2}],y], & & z_{1}\in\mathcal{S}_{k_{1}}(x,Y), \ \ y_{2}\in Y^{k_{2}}, \ \ y\in Y, & & k_{1}+k_{2}+1=k, \\
z & =[[y_{1},y_{2}],x], & & y_{i}\in Y^{k_{i}}, & & k_{1}+k_{2}+1=k,
\end{align*}
we are done by the inductive hypothesis and the definitions of $\mathcal{S}_{k},\mathcal{T}_{k}$. By anti-commutativity, the only case left is
\begin{align*}
z & =[[y_{1},y_{2}],[y_{3},y_{4}]], & & k_{1}+k_{2}+k_{3}+k_{4}=k,
\end{align*}
with exactly one index $i$ for which $y_{i}\in\mathcal{S}_{k_{i}}(x,Y)$ and $y_{j}\in Y^{k_{j}}$ for all $j\neq i$.

We build a second induction on $k_{3}+k_{4}$: by anti-commutativity we may assume that $k_{3}+k_{4}\geq k_{1}+k_{2}$. Using anti-commutativity and either the Jacobi identity or \eqref{eq:malcev}, we rewrite $z$ as
\begin{equation}\label{eq:xspanlie}
[[[y_{1},y_{2}],y_{3}],y_{4}]-[[[y_{1},y_{2}],y_{4}],y_{3}]
\end{equation}
if $\mathfrak{g}$ is a Lie algebra, and as
\begin{equation}\label{eq:xspanmal}
[[[y_{1},y_{3}],y_{2}],y_{4}]+[[[y_{3},y_{2}],y_{4}],y_{1}]+[[[y_{2},y_{4}],y_{1}],y_{3}]+[[[y_{4},y_{1}],y_{3}],y_{2}]
\end{equation}
if $\mathfrak{g}$ is a Mal'cev algebra. If $k_{3}+k_{4}=2$ then $y_{i}=x$ and $y_{j}\in Y$ for all $j\neq i$, and we are done by definition; hence, assume that $[[y'_{1},y'_{2}],[y'_{3},y'_{4}]]\in\mathrm{span}_{K}(\mathcal{T}_{k}(x,Y))$ whenever $k'_{1}+k'_{2}+k'_{3}+k'_{4}=k$ and $k'_{3}+k'_{4}<k_{3}+k_{4}$.

Each $[[y_{j_{1}},y_{j_{2}}],y_{j_{3}}]$ is in the span of either $\mathcal{T}_{\leq k_{j_{1}}+k_{j_{2}}+k_{j_{3}}}(x,Y)$ (if $i$ is among the indices) or $\mathcal{T}_{\leq k_{j_{1}}+k_{j_{2}}+k_{j_{3}}}(Y)$ (if it is not) by main induction, so both \eqref{eq:xspanlie} and \eqref{eq:xspanmal} give
\begin{equation*}
z\in\mathrm{span}_{K}([\mathcal{T}_{\leq k-k_{i}}(Y),y_{i}])+\sum_{j\neq i}\mathrm{span}_{K}([\mathcal{T}_{\leq k-k_{j}}(x,Y),y_{j}]).
\end{equation*}
Focus separately on each term of the sum. If $y_{j}\in Y$ (or $y_{i}=x$), the whole term is in $\mathrm{span}_{K}(\mathcal{T}_{\leq k}(x,Y))$ by definition. If $y_{j}=[y_{j1},y_{j2}]$, we are done by second induction, since $k_{j}<k_{3}+k_{4}$. This concludes both inductions.
\end{proof}

\begin{remark}\label{re:associative}
The same proof as in Proposition~\ref{pr:liemalcev} shows that associative algebras (with $[\cdot,\cdot]$ being the multiplication) are tower algebras. We just use commutativity and $[[y_{1},y_{2}],[y_{3},y_{4}]]=[[[y_{1},y_{2}],y_{3}],y_{4}]$, instead of anti-commutativity and \eqref{eq:xspanlie}.

By the Wedderburn-Artin theorem (already in its initial form by Wedderburn \cite{Mac08}, see for instance \cite[\S 5.2]{Coh03}), every finite-dimensional simple associative algebra is isomorphic to some full matrix algebra $\mathrm{Mat}_{n}(D)$ where $D$ is a division algebra over $K$.
\end{remark}

\begin{remark}\label{re:liesuper}
The proof of Proposition~\ref{pr:liemalcev} applies also to \textit{Lie superalgebras}, i.e.\ $\mathbb{Z}/2\mathbb{Z}$-graded algebras satisfying $[a,b]=-(-1)^{|a||b|}[b,a]$ and $[a,[b,c]]=[[a,b],c]+(-1)^{|a||b|}[b,[a,c]]$, where $|\cdot|$ is the grading. In fact, since we only care about spans, we can rewrite \eqref{eq:xspanlie} appropriately and swap entries of any $[a,b]$ regardless of the possible sign changes.

See \cite[Thm.~2]{Kac77} for a classification of finite-dimensional simple Lie superalgebras over $K=\mathbb{C}$.
\end{remark}

\subsection{Sum-product theorems}\label{se:prelfields}

Sum-product estimates are employed only at the very end of the process to prove Theorem~\ref{th:main}. Moreover, the proof is not overly sensitive to the precise shape of the sets for which we know lower bounds: the procedure works much in the same way, whether we have bounds for $\max\{|X+X|,|XX|\}$ or $|XX+XX|$ or $|(X+X)X|$ or $|(X+X)(X+X)|$ or many other possibilities.

The literature on sum-product estimates is rich and rapidly evolving, and the numerical values of $k,\varepsilon$ in the final bounds change depending on which result we employ. We present here one such estimate, chosen for its large exponent and its generality. The reader interested in producing quantitatively stronger versions of our main theorem can consult for example \cite{MPRRS19} and the references therein.

\begin{theorem}[\cite{RRS16}, Cor.~4]\label{th:sumprodF}
There is an absolute constant $C>0$ such that the following holds. Let $K$ be a field with $\mathrm{char}(K)=p$, and let $X$ be a finite subset of $K$. Then
\begin{align*}
|XX+XX+XX| & \geq C^{-1}|X|^{\frac{7}{4}} & & \text{if $p=0$,} \\
|XX+XX+XX| & \geq C^{-1}\min\{|X|^{\frac{7}{4}},p\} & & \text{if $p>0$.}
\end{align*}
\end{theorem}

Although the field of interest in the paper cited above is $\mathbb{F}_{q}$ with odd characteristic, the result does indeed apply to all fields as observed in \cite[Rem.~5]{RRS16}.

For general fields, we could have also used the estimate on $|XX+XX|$ contained in \cite[Cor.~4]{RRS16} or \cite[Cor.~15]{Rud18} with exponent $\frac{3}{2}$, or the estimate on $\max\{|XX+XX|,|(X+X)(X+X)|\}$ in \cite[Cor.~2.3]{AMRS17} with exponent $\frac{8}{5}$. Increasing the number of copies of $XX$ in the sum does not lead to large improvement: as observed in \cite[Rem.~33]{Shk18}, valid for $\mathbb{F}_{p}$, $m$ copies of $XX$ can give us estimates with exponent $2-\frac{1}{2^{m-1}}$, but $2$ is a hard limit.

For the field $\mathbb{R}$ specifically, other stronger estimates exist. We could have used for instance the estimate on $|X(X-X)|$ in \cite[Thm.~1.14]{Sha19} with exponent $\frac{3}{2}+\frac{7}{226}-o(1)$, or the estimate on $|XX+XX|$ in \cite[Thm.~2]{RS20} with exponent $\frac{3}{2}+\frac{7}{80}-o(1)$.

As for sum-product theorems in their classical form, i.e.\ lower bounds for $\max\{|X+X|,|XX|\}$, the strongest published estimate to date is \cite[Thm.~1.2]{RSS20}, with exponent $\frac{11}{9}$; see also the more recent \cite{MS21}, where the exponent is $\frac{5}{4}$. For the field $\mathbb{R}$, the strongest published estimate to date is instead \cite[Thm.~1]{RS20}, with exponent $\frac{4}{3}+\frac{2}{1167}-o(1)$.

When the field is finite and the set is very large, we also know how to fill the field quickly. This is illustrated in the following result.

\begin{theorem}[\cite{HIKR11}, Thm.~2.5]\label{th:sumprodQ}
Let $X$ be a subset of $\mathbb{F}_{q}$ with $\mathrm{char}(\mathbb{F}_{q})\neq 2$. Let $|X|>q^{\frac{1}{2}\left(1+\frac{1}{d}\right)}$ for some $d\geq 1$. Then $\mathbb{F}_{q}^{*}$ is contained in $dXX$ (the sum of $d$ copies of $XX$).
\end{theorem}

Theorems~\ref{th:sumprodF} and~\ref{th:sumprodQ} show that we have growth in $\mathbb{F}_{q}$ in two cases in which $X$ is far away from being a subfield of $\mathbb{F}_{q}$, namely the two extremes $|X|<p$ and $|X|>\sqrt{q}$. This is a general phenomenon: either $X$ grows under sum or product in a field $K$, or $X$ is too small, or $X$ is contained in $xK'\cup X'$ for a subfield $K'$ of size comparable to $X$ and for a set $X'$ small with respect to $X$ (see \cite[Thm.~2.55]{TV06} and the ensuing discussion). We are not able to retain such precise information on $X$ in the course of our proof, but we can still say something about its size.

\begin{theorem}\label{th:sumprodAv}
There is an absolute constant $\gamma>0$ such that the following holds. For any $\varepsilon>0$ there is $\varepsilon'=\gamma\varepsilon$ (resp.\ for any $\varepsilon'>0$ there is $\varepsilon=\gamma\varepsilon'$) such that, for any finite subset $X$ of a field $K$ with $|X|$ large enough with respect to $\varepsilon$ (resp.\ $\varepsilon'$), either
\begin{enumerate}[(a)]
\item\label{th:sumprodAvgrowth} $\max\{|X+X|,|XX|\}\geq|X|^{1+\varepsilon}$ or
\item\label{th:sumprodAvfield} there is a finite subfield $K'\leq K$ with $|X|^{1-\varepsilon'}\leq|K'|\leq|X|^{1+\varepsilon'}$.
\end{enumerate}
\end{theorem}

\begin{proof}
Assume that \eqref{th:sumprodAvgrowth} does not hold. We can assume that $\varepsilon$ is small enough, by taking $\gamma$ large. Apply \cite[Thm.~2.55]{TV06}: for some absolute constant $C_{2}$, either $|X|\leq C_{2}|X|^{C_{2}\varepsilon}$ (which cannot hold for $\varepsilon$ small and $X$ large), or there is a subfield $K'$ with $|K'|\leq C_{2}|X|^{1+C_{2}\varepsilon}$ and $|X|\leq|K'|+C_{2}|X|^{C_{2}\varepsilon}$. Then for $X$ large there is an appropriate $\varepsilon'$ as in \eqref{th:sumprodAvfield}. The reverse implication follows similarly.
\end{proof}

\section{The escape argument}\label{se:esc}

As we said in \S\ref{se:introover}, we replicate here the escape from subvarieties used in the literature for groups. In the language of algebraic geometry, we need to escape in a few situations: a general variety of degree $1$ (\S\S\ref{se:escgen}-\ref{se:escturr}), and some particular varieties of degree $\leq 2$ (Lemma~\ref{le:quadresc}).

To make a set $A$ ``escape'' from $V$ means to find some small $k$ such that $A^{k}\not\subseteq V$: the original concept goes back to \cite{EMO05}. In particular, if $V$ is a vector subspace we only need to show that the span of $A^{k}$ has dimension larger than $V$: this principle has already been used in \cite[Prop. 3.5]{BDH21}, building on Shitov \cite{Shi19}. Unlike the growth of $A$, the growth of the dimension of the span of $A$ has already been studied in the literature: see most recently \cite{Shi19} \cite{GK20} \cite{Gre21}. The results of this section are in the same spirit as those of \cite{GK20} (for \S\ref{se:escgen}) and \cite{GK22} (for \S\ref{se:escturr}); our versions are more technical though, given the use we make of them later.

\subsection{Span growth for general algebras}\label{se:escgen}

We start with the case of general algebras.

\begin{proposition}\label{pr:span}
Let $\mathfrak{g}$ be an algebra over a field $K$ of dimension $D<\infty$, not necessarily associative, and let $A\subseteq\mathfrak{g}$ be a symmetric set such that $\mathrm{span}_{K}(\langle A\rangle)=\mathfrak{g}$.

Then, for any $d\geq 1$, the set $A^{[\leq s(d)]}$ spans a $K$-vector space of dimension $\geq\min\{d,D\}$, with $s(d)=2^{d-1}$.
\end{proposition}

\begin{proof}
Define $W_{k}:=\mathrm{span}_{K}(A^{[\leq k]})$ for all $k\geq 1$: we have $W_{k}\subseteq W_{k+1}$, so it is sufficient to prove the statement for all $1\leq d\leq D$. Furthermore $W_{k}\subsetneq W_{k+1}$ implies $\dim(W_{k+1})>\dim(W_{k})$ and, since $\dim(W_{k})\leq D<\infty$ for all $k$, there must be some $k_{\max}$ such that $W_{k'}=W_{k_{\max}}$ for all $k'\geq k_{\max}$. Then, \eqref{eq:braspan} implies that $W_{k_{\max}}=\mathrm{span}_{K}(\langle A\rangle)=\mathfrak{g}$.

We argue by induction: since $\mathrm{span}_{K}(\langle A\rangle)=\mathfrak{g}$, there is a nonzero element in $A$ and $s(1)=1$, so the base case is trivial. Now suppose that the result holds up to $d-1$, and assume that $\dim(W_{2^{d-1}})<d$: this means in particular that
\begin{equation*}
W_{2^{d-2}}=W_{2^{d-2}+1}=W_{2^{d-2}+2}=\ldots=W_{2\cdot 2^{d-2}}.
\end{equation*}
But then Lemma~\ref{le:logspan} yields $D=\dim(W_{k_{\max}})<d$, which is a contradiction. Therefore $\dim(W_{2^{d-1}})\geq d$ as well, proving the inductive step.
\end{proof}

We need the result above also for sets built out of a fixed element $v$. Simplicity plays a key role here.

\begin{corollary}\label{co:vspan}
Let $\mathfrak{g}$ be a simple algebra over a field $K$ of dimension $D<\infty$, not necessarily associative, and let $A\subseteq\mathfrak{g}$ be a symmetric set such that $\mathrm{span}_{K}(\langle A\rangle)=\mathfrak{g}$. Let $v\neq 0$ be an element of $\mathfrak{g}$.

Then, for any $d\geq 1$, the set $\mathcal{T}_{\leq d}(v,A^{[\leq s(D)]})$ spans a $K$-vector space of dimension $\geq\min\{d,D\}$, with $s(D)$ as in Proposition~\ref{pr:span}.
\end{corollary}

\begin{proof}
The bilinearity of $[\cdot,\cdot]$ implies directly that for all sets $X,Y\subseteq\mathfrak{g}$ we have
\begin{equation}\label{eq:spanbra}
[\mathrm{span}_{K}(X),\mathrm{span}_{K}(Y)]=\mathrm{span}_{K}([X,Y]).
\end{equation}
Define $V_{k}:=\mathrm{span}_{K}(\mathcal{T}_{\leq k}(v,A^{[\leq s(D)]}))$ for all $k\geq 1$: we have $V_{k}\subseteq V_{k+1}$, so it is sufficient to prove the statement for all $1\leq d\leq D$.

We argue by induction. For $d=1$ the statement is trivial since $v\neq 0$. Suppose now that the result holds for $d-1$. The set $A^{[\leq s(D)]}$ spans the whole $\mathfrak{g}$ by Proposition~\ref{pr:span}, and \eqref{eq:spanbra} yields
\begin{equation*}
V_{d}\supseteq\mathrm{span}_{K}([V_{d-1},A^{[\leq s(D)]}])\cup\mathrm{span}_{K}([A^{[\leq s(D)]},V_{d-1}])=[V_{d-1},\mathfrak{g}]\cup[\mathfrak{g},V_{d-1}].
\end{equation*}
As $\mathfrak{g}$ is simple, we must have either $V_{d-1}=\mathfrak{g}$ or $[V_{d-1},\mathfrak{g}]\cup[\mathfrak{g},V_{d-1}]\not\subseteq V_{d-1}$. In the first case we get $V_{d}=\mathfrak{g}$ and we are done. In the second case, since we already know that $V_{d}\supseteq V_{d-1}$, we obtain that $V_{d}\supsetneq V_{d-1}$ concluding the inductive step.
\end{proof}

\subsection{Span growth for tower algebras}\label{se:escturr}

When $\mathfrak{g}$ is a tower algebra, the process of \S\ref{se:escgen} becomes quicker.

\begin{proposition}\label{pr:spanturr}
Let $\mathfrak{g}$ be a tower algebra over a field $K$ of dimension $D<\infty$, and let $A\subseteq\mathfrak{g}$ be a symmetric set such that $\mathrm{span}_{K}(\langle A\rangle)=\mathfrak{g}$.

Then, for any $d\geq 1$, the set $\mathcal{T}_{\leq d}(A)$ spans a $K$-vector space of dimension $\geq\min\{d,D\}$.
\end{proposition}

\begin{proof}
Define $W_{k}:=\mathrm{span}_{K}(\mathcal{T}_{\leq k}(A))$ for all $k\geq 1$. Arguing as in Proposition~\ref{pr:span}, we reduce to the case in which $1\leq d\leq D$, the result holds for $d-1$, and $\dim(W_{d})<d$. This must imply that $W_{d}=W_{d-1}$, and in particular that $[W_{d-1},A]\cup[A,W_{d-1}]\subseteq W_{d-1}$. Then
\begin{align*}
W_{d+1} & =\mathrm{span}_{K}(\mathcal{T}_{\leq d+1}(A))=\mathrm{span}_{K}(\mathcal{T}_{\leq d}(A)\cup\mathcal{T}_{d+1}(A)) \\
 & \subseteq\mathrm{span}_{K}(W_{d}\cup[W_{d},A]\cup[A,W_{d}]) \\
 & =\mathrm{span}_{K}(W_{d-1}\cup[W_{d-1},A]\cup[A,W_{d-1}])\subseteq W_{d-1}.
\end{align*}
Hence, $W_{d'}=W_{d-1}$ for all $d'\geq d$.

As already said, \eqref{eq:braspan} implies that $\mathfrak{g}=\mathrm{span}_{K}(A^{[\leq d']})$ for $d'$ large enough. However, since $\mathfrak{g}$ is a tower algebra we get in particular $A^{[\leq d']}\subseteq W_{d'}$; therefore $\mathfrak{g}=W_{d'}$ for $d'$ large enough, and then $D<d$, which is a contradiction. Thus, $\dim(W_{d})\geq d$.
\end{proof}

\begin{corollary}\label{co:vspanturr}
Let $\mathfrak{g}$ be a simple tower algebra over a field $K$ of dimension $D<\infty$, and let $A\subseteq\mathfrak{g}$ be a symmetric set such that $\mathrm{span}_{K}(\langle A\rangle)=\mathfrak{g}$. Let $v\neq 0$ be an element of $\mathfrak{g}$.

Then, for any $d\geq 1$, the set $\mathcal{T}_{\leq d+D}(v,A)$ spans a $K$-vector space of dimension $\geq\min\{d,D\}$.
\end{corollary}

\begin{proof}
Define $V_{k}:=\mathrm{span}_{K}(\mathcal{T}_{\leq k+D}(v,A))$ for all $k\geq 1$. Arguing as in Corollary~\ref{co:vspan}, we reduce to the case in which $1\leq d\leq D$ and the result holds for $d-1$. If $V_{d-1}\subsetneq V_{d}$ or $V_{d}=\mathfrak{g}$ we are done, so assume that $V_{d-1}=V_{d}\subsetneq\mathfrak{g}$. By Proposition~\ref{pr:spanturr}, we have $\mathrm{span}_{K}(\mathcal{T}_{\leq d+D}(A))=\mathrm{span}_{K}(\mathcal{T}_{\leq d+D-1}(A))=\mathfrak{g}$. Similarly to the proof of Proposition~\ref{pr:spanturr}, we get
\begin{align*}
V_{d+1} & \subseteq\mathrm{span}_{K}(V_{d}\cup[V_{d},A]\cup[A,V_{d}]\cup[\mathcal{T}_{\leq d+D}(A),v]\cup[v,\mathcal{T}_{\leq d+D}(A)]) \\
 & =\mathrm{span}_{K}(V_{d-1}\cup[V_{d-1},A]\cup[A,V_{d-1}]\cup[\mathcal{T}_{\leq d+D-1}(A),v]\cup[v,\mathcal{T}_{\leq d+D-1}(A)]) \\
 & \subseteq V_{d}
\end{align*}
and then $V_{d'}=V_{d-1}$ for all $d'\geq d$. As $\mathrm{span}_{K}(A^{D})\supseteq\mathrm{span}_{K}(\mathcal{T}_{\leq D}(A))=\mathfrak{g}$ again by Proposition~\ref{pr:spanturr} and $\mathfrak{g}$ is a tower algebra, we obtain in particular
\begin{align*}
\mathrm{span}_{K}([\mathfrak{g},V_{d}]\cup[V_{d},\mathfrak{g}]) & =\mathrm{span}_{K}([A^{D},V_{d}]\cup[V_{d},A^{D}])\subseteq\mathrm{span}_{K}(\mathcal{S}_{\leq d+2D}(v,A)) \\
 & \subseteq\mathrm{span}_{K}(\mathcal{T}_{\leq d+2D}(v,A))=V_{d+D}=V_{d}.
\end{align*}
This contradicts the simplicity of $\mathfrak{g}$. Thus, $\dim(V_{d})\geq\min\{d,D\}$.
\end{proof}

\begin{remark}\label{re:vspanlie}
If $\mathfrak{g}$ is a Lie algebra, by setting up a more elaborate induction one can improve Corollary~\ref{co:vspanturr} to hold with $\mathcal{T}_{\leq d}(v,A)$ instead of $\mathcal{T}_{\leq d+D}(v,A)$. This leads to quantitative (but not qualitative) improvements in the final bounds.
\end{remark}

\subsection{Other escape results}\label{se:escmore}

Thanks to Proposition~\ref{pr:spanturr} we can quickly escape from a $K$-vector subspace. We conclude by proving an escape (or rather, a non-escape) result for two specific expressions. We need this result only in \S\ref{se:genericgoodmap} for simple Lie algebras, but the proof is completely general. In algebraic-geometric terms, it is just an application of B\'ezout's theorem.

\begin{lemma}\label{le:quadresc}
Let $\mathfrak{g}$ be an algebra over a field $K$. Let $x,y,z_{1},z_{2}$ be nonzero elements of $\mathfrak{g}$.
\begin{enumerate}[(a)]
\item\label{le:quadresclin} Define $f:\mathfrak{g}\rightarrow\mathfrak{g}$ as $f(z)=[[z,x],y]$. If there are two distinct $k,k'\in K$ such that $f(z_{1}+kz_{2})=f(z_{1}+k'z_{2})=0$, then $f(Kz_{1}+Kz_{2})=\{0\}$.
\item\label{le:quadrescquadr} Define $g:\mathfrak{g}\rightarrow\mathfrak{g}$ as $g(z)=[[z,x],[z,y]]$. If there are three distinct $k,k',k''\in K$ such that $g(z_{1}+kz_{2})=g(z_{1}+k'z_{2})=g(z_{1}+k''z_{2})=0$, then $g(Kz_{1}+Kz_{2})=\{0\}$.
\item\label{le:quadresctot} Let $\dim(\mathfrak{g})<\infty$, and let $\{s_{i}\}_{i}$ be a basis of $\mathfrak{g}$ as a $K$-vector space. If $g(z)=0$ for all $z=s_{i}$ and all $z=s_{i}+s_{j}$, then $g(\mathfrak{g})=\{0\}$.
\end{enumerate}
\end{lemma}

\begin{proof}
Let $a,b\in K$. We can write
\begin{equation*}
az_{1}+bz_{2}=\lambda(z_{1}+kz_{2})+\mu(z_{1}+k'z_{2})
\end{equation*}
with $\lambda=\frac{b-ak'}{k-k'}$ and $\mu=\frac{b-ak}{k'-k}$. Using bilinearity twice and the hypothesis, we obtain
\begin{equation*}
f(az_{1}+bz_{2})=\lambda f(z_{1}+kz_{2})+\mu f(z_{1}+k'z_{2})=0,
\end{equation*}
and \eqref{le:quadresclin} follows.

Now let us write
\begin{equation*}
z_{1}+k''z_{2}=t(z_{1}+kz_{2})+(1-t)(z_{1}+k'z_{2})
\end{equation*}
with $t=\frac{k''-k'}{k-k'}\in K\setminus\{0,1\}$. By bilinearity and the hypothesis we have then
\begin{align*}
0 & =g(z_{1}+k''z_{2})-t^{2}g(z_{1}+kz_{2})-(1-t)^{2}g(z_{1}+k'z_{2}) \\
 & =t(1-t)([[z_{1}+kz_{2},x],[z_{1}+k'z_{2},y]]+[[z_{1}+k'z_{2},x],[z_{1}+kz_{2},y]]),
\end{align*}
which implies
\begin{equation*}
[[z_{1}+kz_{2},x],[z_{1}+k'z_{2},y]]=-[[z_{1}+k'z_{2},x],[z_{1}+kz_{2},y]].
\end{equation*}
Then, much like before we obtain
\begin{align*}
g(az_{1}+bz_{2})= & \ \lambda^{2}g(z_{1}+kz_{2})+\mu^{2}g(z_{1}+k'z_{2}) \\
 & \ +\lambda\mu([[z_{1}+kz_{2},x],[z_{1}+k'z_{2},y]]+[[z_{1}+k'z_{2},x],[z_{1}+kz_{2},y]]).
\end{align*}
All the summands on the right hand side are $0$, and \eqref{le:quadrescquadr} follows.

Finally, we have
\begin{equation*}
[[s_{i},x],[s_{j},y]]+[[s_{j},x],[s_{i},y]]=g(s_{i}+s_{j})-g(s_{i})-g(s_{j})=0.
\end{equation*}
Let $z=\sum_{i}a_{i}s_{i}$ with $a_{i}\in K$. We reason by induction on the number of nonzero $a_{i}$. If there is one nonzero $a_{i}$, we are already done by the hypothesis. Up to reordering the basis elements, say that $z=\sum_{i=1}^{k}a_{i}s_{i}$ with $k\geq 2$, and assume that the statement is proved for $k-1$. Then
\begin{align*}
g(z) & =g\left(\sum_{i<k}a_{i}s_{i}\right)+\sum_{i<k}[[a_{i}s_{i},x],[a_{k}s_{k},y]]+\sum_{i<k}[[a_{k}s_{k},x],[a_{i}s_{i},y]]+g(a_{k}s_{k}) \\
 & =\sum_{i<k}a_{i}a_{k}([[s_{i},x],[s_{k},y]]+[[s_{k},x],[s_{i},y]])=0,
\end{align*}
and the inductive step is proved.
\end{proof}

\section{Dimensional estimates}\label{se:dimest}

The concept of dimensional estimate goes back to \cite{LP11}, and it rapidly acquired a central position in growth results for groups of Lie type: see \cite[Thm.~4.1]{BGT11}, \cite[Thm.~40]{PS16}, and \cite[Thms.~4.3-5.6]{BDH21}. We provide here an estimate of our own, for simple and tower algebras: the main result of this section is Theorem~\ref{th:main0}.

Let us fix some lexicon. We have already defined linear affine spaces. A function $f:K^{m}\rightarrow K^{m'}$ between $K$-vector spaces is a \textit{linear affine map} if it is defined by polynomials of degree $1$ (not necessarily homogeneous); we use the term \textit{linear homogeneous map} to indicate that the polynomials are also homogeneous. Equivalently, $f$ is linear homogeneous if $f(k_{1}x_{1}+k_{2}x_{2})=k_{1}f(x_{1})+k_{2}f(x_{2})$ for all $x_{1},x_{2}\in K^{m}$ and all $k_{1},k_{2}\in K$. Given any function $f:S\rightarrow T$ between two sets, a subset $S'\subseteq S$ and an element $t\in T$, the \textit{fibre of $t$ in $S'$ through $f$} is the set
\begin{equation*}
f^{-1}(t)\cap S'=\{s\in S'|f(s)=t\}.
\end{equation*}

The next two results barely need proof. Note how elementary Proposition~\ref{pr:linearmap} is when compared for instance to \cite[Cor.~A.3]{BDH21}.

\begin{proposition}\label{pr:linearmap}
Let $X\subseteq K^{m}$ be a linear affine space over a field $K$, and let $f:K^{m}\rightarrow K^{m'}$ be a linear affine map. Then $f(X)$ is a linear affine space of $K^{m'}$, and every nonempty fibre in $X$ through $f$ is a linear affine space $Y$ of $K^{m}$ with $\dim(Y)=\dim(X)-\dim(f(X))$.
\end{proposition}

\begin{proof}
Up to composition with appropriate translations, we may assume that $X$ is a $K$-vector subspace and that $f$ is linear homogeneous. Then every $K$-linear combination of points in $f(X)$ is contained in $f(X)$, meaning that $f(X)$ is a vector subspace. The set $f^{-1}(0)$ is also a vector subspace, and so is the fibre $f^{-1}(0)\cap X$, which has dimension $\dim(X)-\dim(f(X))$ by the rank-nullity theorem. Finally, for any $y\in f(X)$ fix $x_{0}\in X$ such that $f(x_{0})=y$: since $X+x_{0}=X$ we obtain
\begin{align*}
f^{-1}(y)\cap X & =\{x\in X|f(x)=y\}=\{x\in X|f(x-x_{0})=f(x)-y=0\} \\
 & =\{x\in X|f(x)=0\}+x_{0}=(f^{-1}(0)\cap X)+x_{0},
\end{align*}
so every nonempty fibre is a linear affine space too and has the same dimension $\dim(X)-\dim(f(X))$.
\end{proof}

\begin{proposition}\label{pr:imfibre}
Let $X\subseteq K^{m}\times K^{m}$ be a linear affine space over a field $K$, and let $f:K^{m}\times K^{m}\rightarrow K^{m}$ be a linear affine map. Let $A\subseteq K^{m}$ be a finite set such that $f(A^{t_{1}}\times A^{t_{2}})\subseteq A^{k}$ for certain constants $t_{1},t_{2},k$. Then
\begin{equation*}
|(A^{t_{1}}\times A^{t_{2}})\cap X|\leq|A^{k}\cap f(X)|\cdot|(A^{t_{1}}\times A^{t_{2}})\cap W|,
\end{equation*}
where $f(X)$ is a linear affine space, and where $W=\{x\in X|f(x)=y\}$ for some $y\in K^{m}$, with $W$ a linear affine space and with $\dim(W)=\dim(X)-\dim(f(X))$.
\end{proposition}

\begin{proof}
By definition we have
\begin{equation*}
(A^{t_{1}}\times A^{t_{2}})\cap X=\bigcup_{x\in A^{k}\cap f(X)}(A^{t_{1}}\times A^{t_{2}})\cap f^{-1}(x)\cap X,
\end{equation*}
so that
\begin{align*}
|(A^{t_{1}}\times A^{t_{2}})\cap X| & =\sum_{x\in A^{k}\cap f(X)}|(A^{t_{1}}\times A^{t_{2}})\cap f^{-1}(x)\cap X| \\
 & \leq|A^{k}\cap f(X)|\cdot\max_{x\in A^{k}\cap f(X)}|(A^{t_{1}}\times A^{t_{2}})\cap f^{-1}(x)\cap X| \\
 & =|A^{k}\cap f(X)|\cdot|(A^{t_{1}}\times A^{t_{2}})\cap W|
\end{align*}
where $W$ is the fibre $f^{-1}(x)\cap X$ having largest intersection with $A^{t_{1}}\times A^{t_{2}}$. Since $X$ and $f$ are linear affine, the rest follows from Proposition~\ref{pr:linearmap}.
\end{proof}

The next result is the dimensional estimate we were aiming for: Theorem~\ref{th:main0} is just Theorem~\ref{th:dimest} with $t=1$.

\begin{theorem}\label{th:dimest}
Let $\mathfrak{g}$ be a simple algebra over a field $K$ of dimension $D<\infty$, and let $A\subseteq\mathfrak{g}$ be a symmetric set with $\mathrm{span}_{K}(\langle A\rangle)=\mathfrak{g}$. Let $V\subseteq\mathfrak{g}$ be a linear affine space.

Then
\begin{equation*}
|A^{t}\cap V|\leq|A^{k}|^{\frac{\dim(V)}{D}}
\end{equation*}
for any $t\geq 1$, with
\begin{enumerate}[(a)]
\item\label{th:dimestall} $k=tD+D(D-1)2^{D-2}$, and
\item\label{th:dimestturr} $k=tD+\frac{3}{2}D(D-1)$ if $\mathfrak{g}$ is a tower algebra.
\end{enumerate}
\end{theorem}

\begin{proof}
Set $\dim(V)=d$. We proceed by induction on $d$, proving that we can choose $k=k(t,D,d)$ as in the statement. If $d=0$ then $V$ is a single point and $|A^{t}\cap V|\leq 1$. If $d=D$ then $V=\mathfrak{g}$ and $|A^{t}\cap V|=|A^{t}|$.

Fix now $0<d<D$, and assume that the theorem holds for all linear affine spaces of dimension $<d$. First, we show how to bound $|A^{t_{1}}\cap V|\cdot|A^{t_{2}}\cap V'|$ where $V'$ is another linear affine space of dimension $d'$ with $d\leq d'<D$.

Take any nonzero element $v\in V$. By Corollary~\ref{co:vspan}, $\mathcal{T}_{\leq d'+1}(v,A^{[\leq 2^{D-1}]})$ spans a $K$-vector space of dimension larger than $d'$; therefore, there must be an element $w\in\mathcal{T}_{\leq d'+1}(v,A^{[\leq 2^{D-1}]})$ such that $\dim(Kw+V')>\dim(V')$. If $\mathfrak{g}$ is also a tower algebra, we can do the same with $w\in\mathcal{T}_{\leq d'+D+1}(v,A)$ by Corollary~\ref{co:vspanturr}.

Since $w$ is written naturally as a tower with one entry equal to $v$, we can define the function $w(x)$ using the same expression but replacing $v$ with $x$: by the bilinearity of $[\cdot,\cdot]$ and the definition of tower, $w(x)$ is a linear homogeneous map. Define also $f:\mathfrak{g}\times\mathfrak{g}\rightarrow\mathfrak{g}$ as $f(x,y)=w(x)+y$: $f$ is again linear homogeneous, and $f(A^{t_{1}}\times A^{t_{2}})\subseteq A^{t_{1}+t_{2}+g(D,d')}$ with $g(D,d')=d'2^{D-1}$, or $g(D,d')=d'+D$ if $\mathfrak{g}$ is a tower algebra.

The product $V\times V'$ is a linear affine space, so we can use Proposition~\ref{pr:imfibre} and obtain
\begin{equation*}
|A^{t_{1}}\cap V|\cdot|A^{t_{2}}\cap V'|\leq|A^{t_{1}+t_{2}+g(D,d')}\cap f(V\times V')|\cdot|(A^{t_{1}}\times A^{t_{2}})\cap W|
\end{equation*}
for $W$ a fibre in $V\times V'$ through $f$. Both $f(V\times V')$ and $W$ are linear affine, and $\dim(f(V\times V'))>d'$ by the choice of $w(x)$; hence, $\dim(W)=d+d'-\dim(f(V\times V'))<d$. Note now that $x$ and $f(x,y)$ determine $y$ uniquely, so given $x\in\mathfrak{g}$ there is at most one $y\in\mathfrak{g}$ for which $(x,y)\in W$. In other words, considering the projection $\pi:\mathfrak{g}\times\mathfrak{g}\rightarrow\mathfrak{g}$ to the first component, we get
\begin{equation*}
|(A^{t_{1}}\times A^{t_{2}})\cap W|\leq|A^{t_{1}}\cap\pi(W)|
\end{equation*}
with $\pi(W)$ again linear affine and $\dim(\pi(W))=\dim(W)$. Putting everything together, renaming $f(V\times V')=V''$ and using the inductive hypothesis on $\pi(W)$, we conclude that
\begin{equation}\label{eq:bigstep}
|A^{t_{1}}\cap V|\cdot|A^{t_{2}}\cap V'|\leq|A^{t_{1}+t_{2}+g(D,d')}\cap V''|\cdot|A^{k(t_{1},D,d_{1})}|^{\frac{d_{1}}{D}},
\end{equation}
where $V''$ is linear affine and where $d_{1}=d+d'-\dim(V'')<d$.

Next we bound $|A^{t}\cap V|$ itself. Start with a quantity $|A^{t}\cap V|^{J+1}$, for some integer $J\geq 0$. We apply \eqref{eq:bigstep} repeatedly to the factors of this quantity, in the following way: setting initially $V_{0}=V$ and $t_{0}=t$, at the $j$-th step we get
\begin{equation}\label{eq:dimjstep}
|A^{t}\cap V|\cdot|A^{t_{j}}\cap V_{j}|\leq|A^{t_{j+1}}\cap V_{j+1}|\cdot|A^{k(t,D,d_{j})}|^{\frac{d_{j}}{D}},
\end{equation}
with $t_{j+1}=t+t_{j}+g(D,\dim(V_{j}))$ and $\dim(V_{j+1})>\dim(V_{j})$, and where $d_{j}=d+\dim(V_{j})-\dim(V_{j+1})$. Since the $V_{j}$ have larger and larger dimensions, the process stops when we cover $\mathfrak{g}$: we assume to have chosen $J$ such that $V_{J}=\mathfrak{g}$, which happens for some $J\leq D-d$. Therefore, putting together \eqref{eq:dimjstep} for all $j\leq J$,
\begin{equation*}
|A^{t}\cap V|^{J+1}\leq|A^{t_{J}}\cap V_{J}|\cdot\prod_{j=0}^{J-1}|A^{k(t,D,d_{j})}|^{\frac{d_{j}}{D}}\leq|A^{t_{J}}|\cdot|A^{k(t,D,d-1)}|^{\frac{\Delta}{D}},
\end{equation*}
where $\Delta=\sum_{j=0}^{J-1}d_{j}=(J+1)d-D$. Hence
\begin{equation*}
|A^{t}\cap V|\leq|A^{k(t,D,d)}|^{\frac{d}{D}}
\end{equation*}
for $k(t,D,d)=\max\{t_{J},k(t,D,d-1)\}$.

It remains to estimate $k$. Since $J\leq D-1$ we have $t_{J}\leq tD+D(D-1)2^{D-2}$ regardless of $d$, and using the inductive hypothesis we obtain
\begin{equation*}
k(t,D,d)=\max\{t_{J},k(t,D,d-1)\}\leq tD+D(D-1)2^{D-2}.
\end{equation*}
If $\mathfrak{g}$ is a tower algebra we get $t_{J}\leq tD+\frac{3}{2}D(D-1)$, and the same for $k(t,D,d)$.
\end{proof}

\section{Generic elements of Lie algebras}\label{se:generic}

From now on, we focus on $\mathfrak{g}$ a finite-dimensional simple Lie algebra. Thanks to Proposition~\ref{pr:liemalcev}, all the results for tower algebras contained in the previous sections apply to $\mathfrak{g}$.

For us, a \textit{classical Lie algebra} is a simple algebra $\mathfrak{g}$ of the form
\begin{align*}
& \mathfrak{sl}_{n}(K) \ \ (n\geq 2), & & \mathfrak{so}_{2n+1}(K) \ \ (n\geq 3), \\ & \mathfrak{sp}_{2n}(K) \ \ (n\geq 2), & & \mathfrak{so}_{2n}(K) \ \ (n\geq 4), \\ & \mathfrak{e}_{6}(K), \ \mathfrak{e}_{7}(K), \ \mathfrak{e}_{8}(K), \ \mathfrak{f}_{4}(K), \ \mathfrak{g}_{2}(K)
\end{align*}
over a certain field $K$. By adopting this name, we follow the usual convention of the literature in the positive characteristic case, which reserves the term ``non-classical'' for algebras not appearing in the complex case (Cartan, Melikian, and others): see for instance \cite[pp.~2-3]{Str17} and \cite[\S 3.1]{PS06}. On the contrary, in the case $K=\mathbb{C}$ the term ``classical'' conventionally covers only the four infinite families $\mathfrak{sl}_{n},\mathfrak{so}_{2n+1},\mathfrak{sp}_{2n},\mathfrak{so}_{2n}$, while $\mathfrak{e}_{6},\mathfrak{e}_{7},\mathfrak{e}_{8},\mathfrak{f}_{4},\mathfrak{g}_{2}$ are ``exceptional'': see for instance \cite[\S 9.4]{FH04}.

The assumptions on $n$ ensure that all the Lie algebras listed above are simple and distinct for $K=\mathbb{C}$ (see \cite[p.~66]{EK13}): for smaller $n$, we have the isomorphisms $\mathfrak{so}_{3}(\mathbb{C})\simeq\mathfrak{sp}_{2}(\mathbb{C})\simeq\mathfrak{sl}_{2}(\mathbb{C})$, $\mathfrak{so}_{5}(\mathbb{C})\simeq\mathfrak{sp}_{4}(\mathbb{C})$, and $\mathfrak{so}_{6}(\mathbb{C})\simeq\mathfrak{sl}_{4}(\mathbb{C})$, while $\mathfrak{so}_{4}(\mathbb{C})\simeq\mathfrak{sl}_{2}(\mathbb{C})\oplus\mathfrak{sl}_{2}(\mathbb{C})$ is not simple (see \cite[\S\S 18.2-19.1]{FH04}). For an arbitrary field $K$, provided that $\mathrm{char}(K)\notin\{2,3\}$, the algebras above are simple when their centre is trivial: this fails to happen only for $\mathfrak{sl}_{n}(K)$ when $\mathrm{char}(K)|n$ (see \cite[pp.~69-71]{EK13} or \cite[\S 3.1]{PS06}).

All the classical Lie algebras are of dimension $>1$ and the Lie bracket is anti-commutative, so $[x,\cdot]$ is not the zero map for any $x\neq 0$ (or else we would violate simplicity).

\subsection{Extremal elements}\label{se:genericextr}

Let $\mathfrak{g}$ be a Lie algebra over a field $K$. An element $x\in\mathfrak{g}$ is \textit{extremal} when $[[\mathfrak{g},x],x]\subseteq Kx$. For every $x\in\mathfrak{g}$ extremal and every $y\in\mathfrak{g}$, we write $\lambda_{x}(y)$ for the value in $K$ such that $[[y,x],x]=\lambda_{x}(y)x$. If $\lambda_{x}(y)=0$ for all $y\in\mathfrak{g}$, $x$ is a \textit{sandwich}.

Extremal and sandwich elements have been extensively studied, especially for their applications in the theory of classification of simple Lie algebras in positive characteristic. For instance, if $K$ is algebraically closed with $\mathrm{char}(K)>5$, classical Lie algebras are the only simple Lie algebras over $K$ without sandwich elements, as proved by Premet \cite{Pre87}.

The following appears without proof in \cite[(2)]{Che89} and \cite[Lemma~2.2]{CSUW01}.

\begin{lemma}\label{le:extremalmani}
Let $\mathfrak{g}$ be a Lie algebra over a field $K$, and let $z\in\mathfrak{g}$ be extremal. Then for every $x,y\in\mathfrak{g}$ we have
\begin{equation*}
2[[x,z],[y,z]]=\lambda_{z}(x)[y,z]-\lambda_{z}(y)[x,z]+\lambda_{z}([x,y])z.
\end{equation*}
\end{lemma}

\begin{proof}
By the properties of the Lie bracket, it is straightforward to obtain
\begin{align*}
2[[x,z],[y,z]] & =[[[y,z],z],x]-[[[y,z],x],z]-[[[x,z],z],y]+[[[x,z],y],z] \\
 & =[[[y,z],z],x]+[[[x,y],z],z]-[[[x,z],z],y]
\end{align*}
and the rest follows by definition of extremal.
\end{proof}

Extremal elements are useful also because they span the Lie algebras we are concerned about. As claimed in the next result, whose proof we partly delay to Appendix~\ref{se:appextr}, there is a basis of $\mathfrak{g}$ as a $K$-vector space made of extremal elements (an \textit{extremal basis}) satisfying also some additional properties.

\begin{theorem}\label{th:extremalbasis}
Let $\mathfrak{g}$ be a classical Lie algebra over a field $K$ with $\mathrm{char}(K)\neq 2$. Then there is an extremal basis $B$ of $\mathfrak{g}$ satisfying the following properties:
\begin{enumerate}[(a)]
\item\label{th:extremalbasisnosand} for every $b\in B$ there is some $z\in\mathfrak{g}$ with $\lambda_{b}(z)\neq 0$ (in other words, $b$ is not a sandwich);
\item\label{th:extremalbasisnofour} there is an element $b_{1}\in B$ such that for every $b\in B\setminus\{b_{1}\}$ there is some $z\in\mathfrak{g}$ with $[[z,b_{1}],[z,b]]\neq 0$.
\end{enumerate}
\end{theorem}

\begin{proof}
By \cite[Prop.~3.3]{CSUW01}, $\mathfrak{g}$ is generated by the set of extremal elements. Then, by \cite[Lemma~2.4]{CSUW01}, $\mathfrak{g}$ is also spanned by the extremal elements. Thus, there exists an extremal basis $B$.

The list of minimal numbers of extremal elements required to generate $\mathfrak{g}$, contained in \cite[Thm.~8.2]{CSUW01}, shows in particular that $\mathfrak{g}$ is finitely generated by them. Thus, by \cite[Cor.~4.5]{CSUW01} there are no sandwich elements in $\mathfrak{g}$, and \eqref{th:extremalbasisnosand} is proved. For \eqref{th:extremalbasisnofour}, see Appendix~\ref{se:appextr}.
\end{proof}

We define a set of \textit{generic} elements, depending on the choice of an extremal basis $B$ of $\mathfrak{g}$ and of an element $b_{1}\in B$:
\begin{equation*}
U=U(B,b_{1}):=\left\{x\in\mathfrak{g}\ \left|\ \begin{array}{ll}\lambda_{b}(x)\neq 0 & \text{for all $b\in B$,} \\ \text{$[[x,b_{1}],[x,b]]\neq 0$} & \text{for all $b\in B\setminus\{b_{1}\}$}\end{array}\right.\right\}.
\end{equation*}
Generic elements in $U$ have the following property, which we shall need for our descent procedure in \S\ref{se:descent}.

\begin{proposition}\label{pr:generic}
Let $\mathfrak{g}$ be a classical Lie algebra over a field $K$ with $\mathrm{char}(K)\neq 2$, and fix an extremal basis $B$ of $\mathfrak{g}$ and an element $b_{1}\in B$. Let $x,y\in\mathfrak{g}$ be two nonzero elements such that $[x,y]=0$ and $[[x,z],[y,z]]=0$ for all $z\in\mathfrak{g}$. Then either $x\notin U$ or $y\in Kx$.
\end{proposition}

\begin{proof}
Assume that $x\in U$. By the hypotheses and Lemma~\ref{le:extremalmani}, for any $b\in B$ we must have
\begin{equation*}
\lambda_{b}(x)[y,b]-\lambda_{b}(y)[x,b]=2[[x,b],[y,b]]=0,
\end{equation*}
and since $\lambda_{b}(x)\neq 0$ we obtain
\begin{equation}\label{eq:lambdaprop}
[y,b]=\frac{\lambda_{b}(y)}{\lambda_{b}(x)}[x,b].
\end{equation}

Using again the hypotheses, we must also have
\begin{align*}
0 & =[[x,b_{1}+b],[y,b_{1}+b]]-[[x,b_{1}],[y,b_{1}]]-[[x,b],[y,b]] \\
 & =[[x,b_{1}],[y,b]]+[[x,b],[y,b_{1}]],
\end{align*}
which together with \eqref{eq:lambdaprop} yields
\begin{equation*}
\left(\frac{\lambda_{b}(y)}{\lambda_{b}(x)}-\frac{\lambda_{b_{1}}(y)}{\lambda_{b_{1}}(x)}\right)[[x,b_{1}],[x,b]]=0.
\end{equation*}
Since $[[x,b_{1}],[x,b]]\neq 0$, this means that $\frac{\lambda_{b}(y)}{\lambda_{b}(x)}=\frac{\lambda_{b_{1}}(y)}{\lambda_{b_{1}}(x)}=:\lambda$ for all $b\in B$. Therefore, since $B$ spans the whole $\mathfrak{g}$, \eqref{eq:lambdaprop} implies that $[y,z]=\lambda[x,z]$ for all $z\in\mathfrak{g}$, or in other words that $[y-\lambda x,\mathfrak{g}]=\{0\}$. Since $[y-\lambda x,\cdot]$ could not be the zero map if we had $y-\lambda x\neq 0$, we conclude that $y=\lambda x$.
\end{proof}

\subsection{Finding a good map}\label{se:genericgoodmap}

The objective of the whole \S\ref{se:generic} is to be able to separate two general non-collinear elements $x,y$ via a linear homogeneous map $f$ so that $[f(x),f(y)]\neq 0$. As we shall see in \S\ref{se:descent}, this is the key step to achieve the dimensional descent we want.

We collected several facts in \S\ref{se:genericextr}, and now we put them together. Theorem~\ref{th:extremalbasis} opens the door to having generic elements, an escape procedure as in \S\ref{se:esc} makes us reach one such element, and Proposition~\ref{pr:generic} shows that we attain a desirable property. Corollary~\ref{co:generic} below sums up these steps.

\begin{corollary}\label{co:generic}
Let $\mathfrak{g}$ be a classical Lie algebra over a field $K$ with $\mathrm{char}(K)=0$ or $\mathrm{char}(K)\geq 3\dim(\mathfrak{g})$, and let $A\subseteq\mathfrak{g}$ be a symmetric set with $\mathrm{span}_{K}(\langle A\rangle)=\mathfrak{g}$. Let $x,y\in\mathfrak{g}$ be two nonzero elements with $y\notin Kx$.

Then there is a linear homogeneous map $f:\mathfrak{g}\rightarrow\mathfrak{g}$ such that
\begin{enumerate}[(a)]
\item $f(A^{t})\subseteq A^{m}$ with $m=O(t\dim(\mathfrak{g})^{2}+\dim(\mathfrak{g})^{3})$ for all $t\geq 1$,
\item $f(x)\neq 0$, and
\item either $f(y)\in Kf(x)$ or $[f(x),f(y)]\neq 0$.
\end{enumerate}
\end{corollary}

\begin{proof}
If $[x,y]\neq 0$ then the identity map is a valid candidate for $f$, so assume $[x,y]=0$. Let $D:=\dim(\mathfrak{g})$. Fix an extremal basis $B=\{b_{1},\ldots,b_{D}\}$ of $\mathfrak{g}$ and an element $b_{1}\in B$ satisfying the conditions of Theorem~\ref{th:extremalbasis}, and define
\begin{align*}
V_{i} & =\{z\in\mathfrak{g}\ |\ \lambda_{b_{i}}(z)=0\}=\{z\in\mathfrak{g}\ |\ [[z,b_{i}],b_{i}]=0\} & & 1\leq i\leq D, \\
V'_{i} & =\{z\in\mathfrak{g}\ |\ [[z,b_{1}],[z,b_{i}]]=0\} & & 2\leq i\leq D.
\end{align*}
By Theorem~\ref{th:extremalbasis}, none of the $V_{i},V'_{i}$ is the whole $\mathfrak{g}$. Moreover, the definition of $U$ yields
\begin{equation*}
\mathfrak{g}=U\sqcup\left(\bigcup_{i=1}^{D}V_{i}\cup\bigcup_{i=2}^{D}V'_{i}\right).
\end{equation*}

By Corollary~\ref{co:vspanturr}, $\mathcal{T}_{\leq 2D}(x,A)$ spans the whole $\mathfrak{g}$: in other words, there is a basis $\{z_{j}\}_{j\leq D}$ of $\mathfrak{g}$ and there are linear homogeneous maps $f_{j}:\mathfrak{g}\rightarrow\mathfrak{g}$ such that $f_{j}(A^{t})\subseteq A^{t+2D-1}$ and $f_{j}(x)=z_{j}$. For notational convenience, we identify integers with their image in $K$ via the (unique) ring homomorphism from $\mathbb{Z}$ to $K$. Then, by the hypothesis on $\mathrm{char}(K)$, the elements of the form
\begin{align*}
& c_{1}z_{1}+c_{2}z_{2}+\ldots+c_{D}z_{D}, & & 0\leq c_{j}\leq 3D-2,
\end{align*}
are all pairwise distinct. Our first goal is to show that among these elements there is at least one $w$ not contained in any of the $V_{i},V'_{i}$ (i.e.\ $w\in U$). We assume otherwise, and prove that this leads to a contradiction.

Start by fixing any $c'_{2},\ldots,c'_{D}$. There are $3D-1$ elements in the set
\begin{equation*}
S_{1}(c'_{2},\ldots,c'_{D})=\{c_{1}z_{1}+c'_{2}z_{2}+\ldots+c'_{D}z_{D}\ |\ 0\leq c_{1}\leq 3D-2\},
\end{equation*}
so there are either at least two elements $w_{1},w_{2}$ in the same $V_{i}$ or at least three elements $w_{1},w_{2},w_{3}$ in the same $V'_{i}$. These elements are collinear, so we may use Lemma~\ref{le:quadresc}\eqref{le:quadresclin} in the first case and Lemma~\ref{le:quadresc}\eqref{le:quadrescquadr} in the second case; thus, the whole line $Kz_{1}+c'_{2}z_{2}+\ldots+c'_{D}z_{D}$ is contained in one of the $V_{i},V'_{i}$.

Now fix $c'_{3},\ldots,c'_{D}$. There are $3D-1$ lines in the set
\begin{equation*}
S_{2}(c'_{3},\ldots,c'_{D})=\{Kz_{1}+c_{2}z_{2}+c'_{3}z_{3}+\ldots+c'_{D}z_{D}\ |\ 0\leq c_{2}\leq 3D-2\},
\end{equation*}
so there are either at least two lines $W_{1},W_{2}$ in the same $V_{i}$ or at least three lines $W_{1},W_{2},W_{3}$ in the same $V'_{i}$. Using Lemma~\ref{le:quadresc}\eqref{le:quadresclin}-\eqref{le:quadrescquadr} pointwise (i.e.\ on pairs or triples of points on such lines), we obtain that the whole $Kz_{1}+Kz_{2}+c'_{3}z_{3}+\ldots+c'_{D}z_{D}$ is contained in one of the $V_{i},V'_{i}$. Iterating the procedure, we conclude that the whole $\mathfrak{g}$ is contained in one of the $V_{i},V'_{i}$, in contradiction with Theorem~\ref{th:extremalbasis}.

Hence, there must be one element $w=c_{1}z_{1}+\ldots+c_{D}z_{D}\in U$; in addition, $w\neq 0$ since $0\notin U$. Let $g:=c_{1}f_{1}+\ldots+c_{D}f_{D}$. We have $g(A^{t})\subseteq A^{D(3D-2)(t+2D-1)}$ and $g(x)=w$. By Proposition~\ref{pr:generic}, either $g(y)=0$ or $[w,g(y)]\neq 0$ or $g(y)\in Kw$ or there is some $z\in\mathfrak{g}$ for which $[[w,z],[g(y),z]]\neq 0$. In all but the last case, the result follows by taking $f:=g$. In the last case, by Lemma~\ref{le:quadresc}\eqref{le:quadresctot} there must be some $z=s_{i}$ or $z=s_{i}+s_{j}$ such that $[[w,z],[g(y),z]]\neq 0$, where $\{s_{i}\}_{i}$ is any basis of $\mathfrak{g}$. By Proposition~\ref{pr:spanturr} there is one such basis inside $\mathcal{T}_{\leq D}(A)\subseteq A^{D}$, so we can define $f:=[g(\cdot),z]$: since $f(A^{t})\subseteq A^{D(3D-2)(t+2D-1)+D}$ and $[f(x),f(y)]\neq 0$ (implying $f(x)\neq 0$ as well), we are done.
\end{proof}

\section{Descent}\label{se:descent}

By Corollary~\ref{co:generic}, we can show that it is possible to reduce the dimension of a vector space and preserve the density of $A$, due to the effects of Theorem~\ref{th:dimest} when going through a linear homogeneous map.

\begin{proposition}\label{pr:descent}
Let $\mathfrak{g}$ be a classical Lie algebra over a field $K$ with $\mathrm{char}(K)=0$ or $\mathrm{char}(K)\geq 3\dim(\mathfrak{g})$, and let $A\subseteq\mathfrak{g}$ be a symmetric set with $\mathrm{span}_{K}(\langle A\rangle)=\mathfrak{g}$. Let $V$ be a $K$-vector subspace of $\mathfrak{g}$ with $\dim(V)>1$.

Then there is a $K$-vector subspace $W$ of $\mathfrak{g}$ with $0<\dim(W)<\dim(V)$ such that
\begin{equation*}
|A^{t}\cap V|\leq|A^{m}\cap W|\cdot|A^{k}|^{\frac{\dim(V)-\dim(W)}{\dim(\mathfrak{g})}}
\end{equation*}
for any $t\geq 1$, with
\begin{align*}
m & =m(t,\dim(\mathfrak{g}))=O(t\dim(\mathfrak{g})^{2}+\dim(\mathfrak{g})^{3}), \\
k & =k(t,\dim(\mathfrak{g}))=O(t\dim(\mathfrak{g})+\dim(\mathfrak{g})^{2}).
\end{align*}
\end{proposition}

\begin{proof}
If $A^{t}\cap V=\{0\}$ the result is immediate, so we assume otherwise. Let $D:=\dim(\mathfrak{g})$ and $d:=\dim(V)$. Our main goal is to find a linear homogeneous map $f:\mathfrak{g}\rightarrow\mathfrak{g}$ such that we have both $f(A^{t}\cap V)\subseteq A^{m}\cap f(V)$ and $0<\dim(f(V))<d$.

Suppose first that $A^{t}\cap V$ spans a proper $K$-vector space $V'$ of $V$, of dimension $0<d'<d$. We can build a basis of $V$ containing a basis of $V'$ as a proper subset (say $s_{1},\ldots,s_{d}$ spanning $V$ and $s_{1},\ldots,s_{d'}$ spanning $V'$); then we take $f$ to be the linear homogeneous map sending $s_{j}$ to itself for $j\leq d'$ and to $0$ for $d'<j\leq d$. This definition yields $f(V)=V'$ with $V'$ stabilized pointwise, and since in this case $A^{t}\cap V=A^{t}\cap V'$ we are done.

Suppose now that $A^{t}\cap V$ spans the whole $V$. Since $\dim(V)>1$, we can fix two nonzero elements $x,y\in A^{t}\cap V$ with $y\notin Kx$. By Corollary~\ref{co:generic}, there is a constant $m_{1}=O(tD^{2}+D^{3})$ and there is a linear homogeneous map $f_{1}$ with $f_{1}(A^{t})\subseteq A^{m_{1}}$, $f_{1}(x)\neq 0$, and either $f_{1}(y)\in Kf_{1}(x)$ or $[f_{1}(x),f_{1}(y)]\neq 0$. In the first case, $f_{1}(y-cx)=0$ for some $c\in K$: this means that $\dim(f(V))<d$ by the rank-nullity theorem (because $y-cx\in V$) and also $\dim(f(V))>0$ (because $x\in V$), so we can take $f:=f_{1}$. In the second case, we define instead $f:=[f_{1}(x),f_{1}(\cdot)]$: this $f$ is linear homogeneous, has $f(A^{t})\subseteq A^{2m_{1}}$, $f(x)=0$ and $f(y)\neq 0$, so by the rank-nullity theorem again $0<\dim(f(V))<d$, and we are done.

Now that we have $f$, we get the bound
\begin{equation*}
|A^{t}\cap V|\leq|A^{m}\cap f(V)|\cdot|A^{t}\cap Y|,
\end{equation*}
where $m=2m_{1}$ and $Y$ is the fibre in $V$ through $f$ whose intersection with $A^{t}$ is largest. The $K$-vector subspace $f(V)$ has the correct dimension by our choice of $f$, and $Y$ is a linear affine space of dimension $d-\dim(f(V))$ by Proposition~\ref{pr:linearmap}. Applying Theorem~\ref{th:dimest}\eqref{th:dimestturr} to $Y$, we conclude the proof.
\end{proof}

Since reducing the dimension preserves the density of $A$, we are able to go down to $\dim(V)=1$. We have the following result.

\begin{theorem}\label{th:onedim}
Let $\mathfrak{g}$ be a classical Lie algebra over a field $K$ with $\mathrm{char}(K)=0$ or $\mathrm{char}(K)\geq 3\dim(\mathfrak{g})$, and let $A\subseteq\mathfrak{g}$ be a symmetric set with $\mathrm{span}_{K}(\langle A\rangle)=\mathfrak{g}$.

Then for $k=e^{O(\dim(\mathfrak{g})\log\dim(\mathfrak{g}))}$ the following holds: for every $\varepsilon>0$, either
\begin{enumerate}[(a)]
\item\label{th:onedimgrowth} $|A^{k}|>|A|^{1+\varepsilon}$, or
\item\label{th:onedimone} there is a vector subspace $V\subsetneq\mathfrak{g}$ with $\dim(V)=1$ such that
\begin{equation*}
|A^{k}\cap V|>|A|^{\frac{1}{\dim(\mathfrak{g})}-\varepsilon}.
\end{equation*}
\end{enumerate}
\end{theorem}

\begin{proof}
Let $D:=\dim(\mathfrak{g})$. We apply Proposition~\ref{pr:descent} repeatedly. Starting with $V_{0}=\mathfrak{g}$ and $t_{0}=1$, we have a sequence of vector subspaces $V_{j}$ such that
\begin{equation*}
|A^{t_{j}}\cap V_{j}|\leq|A^{t_{j+1}}\cap V_{j+1}|\cdot|A^{k_{j+1}}|^{\frac{\dim(V_{j})-\dim(V_{j+1})}{D}},
\end{equation*}
where $0<\dim(V_{j+1})<\dim(V_{j})$ and
\begin{align*}
t_{j+1} & =m(t_{j},D), & k_{j+1} & =k(t_{j},D),
\end{align*}
with the functions $m,k$ as in Proposition~\ref{pr:descent}. By the dimensional constraint we reach some $1$-dimensional $V=V_{J}$ in $J\leq D$ steps, and for that $V_{J}$ we obtain
\begin{align}
|A|=|A^{t_{0}}\cap V_{0}| & \leq|A^{t_{J}}\cap V|\cdot\prod_{j=0}^{J-1}|A^{k_{j+1}}|^{\frac{\dim(V_{j})-\dim(V_{j+1})}{D}} \nonumber \\
 & \leq|A^{t_{J}}\cap V|\cdot|A^{k_{J}}|^{1-\frac{1}{D}}. \label{eq:longdescent}
\end{align}
Set $k=\max\{t_{J},k_{J}\}$. If $|A^{k}|\leq|A|^{1+\varepsilon}$ then
\begin{equation}\label{eq:ed}
|A^{k_{J}}|^{1-\frac{1}{D}}\leq|A^{k}|^{1-\frac{1}{D}}\leq|A|^{(1+\varepsilon)\left(1-\frac{1}{D}\right)}<|A|^{1-\frac{1}{D}+\varepsilon}.
\end{equation}
Combining \eqref{eq:longdescent} and \eqref{eq:ed} we conclude the result for $k$ as above. It remains to estimate it: by the constraint on $J$ we get $t_{J}=e^{O(D\log D)}$ and then $k_{J}=e^{O(D\log D)}$ too, and we are done.
\end{proof}

\section{Sum-bracket theorem}\label{se:main}

We can move now to proving the sum-bracket theorem.

\subsection{Proof of Theorem~\ref{th:main}}

Let $D:=\dim(\mathfrak{g})$. As we already discussed in \S\ref{se:introover}, we divide the proof in several steps: first we reduce to a $1$-dimensional space, then we grow in it via the sum-product theorems, and finally we stick it in linearly independent directions. In all these steps we assume that both $\mathrm{char}(K)$ and $|A|$ are large; a further final step will treat the case when either one is small.
\medskip

\textbf{Step 1: reduction to dimension $1$.} From now until almost the end of the proof, we assume that $\mathrm{char}(K)=0$ or $\mathrm{char}(K)\geq 3D$, and also that $A$ is large enough, in the sense that $|A|>\max\{c_{1},e^{c_{2}D}\}$ for any constants $c_{1},c_{2}$ we may need.

Let $m\geq 1$ be an integer: whenever we want to prove cases \eqref{th:mainzero}, \eqref{th:mainprime}, \eqref{th:mainstuck} of the theorem, we take $m=1$. Let $0<\delta\leq 1$: whenever we want to prove cases \eqref{th:mainzero}, \eqref{th:mainprime}, \eqref{th:mainbig} of the theorem, we take $\delta=1$. Define $\eta=\lceil\frac{3\gamma}{\delta}\rceil$, where $\gamma$ is as in Theorem~\ref{th:sumprodAv}: we can ask for $\gamma$ to be a positive integer, so that in particular $\frac{3\gamma}{\delta}\leq\eta<\frac{4\gamma}{\delta}$.

Apply repeatedly Theorem~\ref{th:onedim} with $\varepsilon=\frac{1}{\eta mD}$. Then, either we fall into case \eqref{th:onedimone} in at most $\eta mD$ steps or, after falling into case \eqref{th:onedimgrowth} $\eta mD$ times, we have obtained
\begin{equation}\label{eq:firstgrowth}
|A^{k^{\eta mD}}|>|A|^{\left(1+\frac{1}{\eta mD}\right)^{\eta mD}}>|A|^{2}
\end{equation}
with $k=e^{O(D\log D)}$. If \eqref{eq:firstgrowth} holds, then we have Theorem~\ref{th:main}\eqref{th:mainzero}-\eqref{th:mainprime}-\eqref{th:mainstuck}; moreover, \eqref{eq:firstgrowth} cannot hold under the hypotheses of Theorem~\ref{th:main}\eqref{th:mainbig} since we already start with $|A|>|\mathfrak{g}|^{\frac{1}{2}}$. Thus, assume instead that we fall into Theorem~\ref{th:onedim}\eqref{th:onedimone} in at most $\eta mD$ steps: this means that
\begin{equation}\label{eq:oneattheend}
|A^{k^{\eta mD}}\cap V|\geq|A^{ks}\cap V|>|A^{s}|^{\left(1-\frac{1}{\eta m}\right)\frac{1}{D}}\geq|A|^{\left(1-\frac{1}{\eta m}\right)\frac{1}{D}}
\end{equation}
for some $1\leq s\leq k^{\eta mD-1}$.

Since $V$ is a $1$-dimensional vector subspace, we have $V=Kv$ for any nonzero $v\in V$. Since $|A|\geq 2$, by \eqref{eq:oneattheend} we may take $v\in A^{k^{\eta mD}}\setminus\{0\}$. As $[v,\cdot]$ is not the zero map, there is some $w\in\mathfrak{g}$ such that $[v,w]\neq 0$; by Corollary~\ref{co:vspanturr} and bilinearity, we may assume that $w\in\mathcal{T}_{\leq 2D}(v,A)$. By anti-commutativity, up to switching signs we can rewrite $w$ to be of the form
\begin{equation*}
w=[\ldots[[\ldots[a_{1},a_{2}]\ldots,a_{i}],v]\ldots,a_{j}]
\end{equation*}
for some $0\leq i\leq j<2D$ and some $a_{1},\ldots,a_{j}\in A$.
\medskip

\textbf{Step 2: growth by sum-product.} Set
\begin{equation*}
X:=\{x\in K|xv\in A^{k^{\eta mD}}\},
\end{equation*}
for which $|X|>|A|^{\left(1-\frac{1}{\eta m}\right)\frac{1}{D}}$ by \eqref{eq:oneattheend}. Then we have
\begin{align*}
(X+X)v & =Xv+Xv\subseteq A^{2k^{\eta mD}}\cap Kv, \\
(XX)[v,w] & =[Xv,[\ldots[[\ldots[a_{1},a_{2}]\ldots,a_{i}],Xv]\ldots,a_{j}]] \\
 & \subseteq A^{2k^{\eta mD}+2D-1}\cap K[v,w],
\end{align*}
and $v,[v,w]\neq 0$. Let us argue now separately, depending on which case of the theorem we want to prove: in each of them, we aim to obtain a $1$-dimensional vector space $Z$ in which either we have growth or we reach every element.
\medskip

\textit{Case~\eqref{th:mainzero}.} Let $z:=[v,w]$ and $Z:=Kz$, and set $k':=2k^{3\gamma D}+2D-1$. Using Theorem~\ref{th:sumprodF} and the fact that $A$ is large enough, we conclude that
\begin{equation*}
|A^{3k'}\cap Z|\geq|XX+XX+XX|\geq C^{-1}|X|^{\frac{7}{4}}\geq C^{-1}|A|^{\frac{7}{6D}}\geq|A|^{\frac{8}{7D}}.
\end{equation*}

\textit{Case~\eqref{th:mainprime}.} Let $z:=[v,w]$ and $Z:=Kz$, and set $k':=2k^{3\gamma D}+2D-1$. Again using Theorem~\ref{th:sumprodF} and $A$ large enough, we get
\begin{align*}
|A^{3k'}\cap Z| & \geq|XX+XX+XX|\geq C^{-1}\min\{|X|^{\frac{7}{4}},p\}\geq C^{-1}\min\{|A|^{\frac{7}{6D}},p\} \\
 & \geq\min\{|A|^{\frac{8}{7D}},C^{-1}p\};
\end{align*}
by Cauchy-Davenport (see for instance \cite[Thm.~5.4]{TV06}), from $|A^{3k'}\cap Z|\geq C^{-1}p$ we also obtain $|A^{6Ck'}\cap Z|\geq p$.

\textit{Case~\eqref{th:mainstuck}.} Let
\begin{align*}
z & :=v & k' & :=2k^{\eta D} & & \text{if $|X+X|\geq|XX|$,} \\
z & :=[v,w] & k' & :=2k^{\eta D}+2D-1 & & \text{otherwise,}
\end{align*}
and let $Z:=Kz$. We already have $|X|>|A|^{\left(1-\frac{1}{\eta}\right)\frac{1}{D}}$; by Theorem~\ref{th:dimest} on the other hand $|X|\leq|A^{\frac{5}{2}Dk^{\eta D}}|^{\frac{1}{D}}$ (if $k\geq 2$, say). If $|A^{\frac{5}{2}Dk^{\eta D}}|\geq|A|^{1+\frac{1}{\eta}}$ we are done, so assume otherwise. Since $\eta\geq\frac{3\gamma}{\delta}$, we have
\begin{align*}
|X|^{1-\frac{\delta}{2}} & >|A|^{\left(1-\frac{\delta}{2}\right)\left(1-\frac{1}{\eta}\right)\frac{1}{D}}>|A|^{\frac{1-\delta}{D}}, \\
|X|^{1+\frac{\delta}{2}} & \leq|A^{\frac{5}{2}Dk^{\eta D}}|^{\left(1+\frac{\delta}{2}\right)\frac{1}{D}}<|A|^{\left(1+\frac{\delta}{2}\right)\left(1+\frac{1}{\eta}\right)\frac{1}{D}}<|A|^{\frac{1+\delta}{D}}.
\end{align*}
By hypothesis, there is no subfield $K'$ with $|A|^{\frac{1-\delta}{D}}\leq|K'|\leq|A|^{\frac{1+\delta}{D}}$; thus, there is no subfield with $|X|^{1-\frac{\delta}{2}}\leq|K'|\leq|X|^{1+\frac{\delta}{2}}$ either. Finally, Theorem~\ref{th:sumprodAv} gives us
\begin{equation*}
|A^{k'}\cap Z|\geq\max\{|X+X|,|XX|\}\geq|X|^{1+\frac{\delta}{2\gamma}}\geq|A|^{\left(1+\frac{\delta}{2\gamma}\right)\left(1-\frac{1}{\eta}\right)\frac{1}{D}}\geq|A|^{\left(1+\frac{\delta}{8\gamma}\right)\frac{1}{D}}
\end{equation*}
since $\eta<\frac{4\gamma}{\delta}$ and $\frac{\gamma}{\delta}\geq 1$.

\textit{Case~\eqref{th:mainbig}.} Let $z:=[v,w]$ and $Z:=Kz$, and set $k':=2k^{3\gamma mD}+2D-1$. By the hypothesis, $|X|>q^{\frac{1}{2}\left(1-\frac{1}{3\gamma m}\right)\left(1+\frac{1}{m}\right)}\geq q^{\frac{1}{2}\left(1+\frac{1}{3m}\right)}$ since $\gamma,m\geq 1$. Using Theorem~\ref{th:sumprodQ}, we conclude that
\begin{equation*}
A^{3mk'}\cap Z\supseteq(3mXX)z\supseteq Z\setminus\{0\},
\end{equation*}
and $A^{3mk'}$ contains also $0$ since $A$ is symmetric.
\medskip

\textbf{Step 3: independent directions.} We now return to treating all cases at the same time. Collecting together all our choices, we have $k'=k^{O(mD/\delta)}=e^{O(mD^{2}\log D/\delta)}$. Set $k''=\max\{3k',6Ck',k',3mk'\}$, and define
\begin{equation*}
Y=\{y\in K|yz\in A^{k''}\}.
\end{equation*}
By Corollary~\ref{co:vspanturr} the set $\mathcal{T}_{\leq 2D}(z,A)\subseteq A^{k'+2D-1}$ contains a basis of $\mathfrak{g}$; moreover, by definition the elements $z_{i}$ of this basis are written in terms of $z$ in such a way that we have $yz_{i}\in\mathcal{T}_{\leq 2D}(yz,A)$ for any $y\in K$. Thus $Yz_{i}\subseteq A^{k''+2D-1}$, and since the $z_{i}$ form a basis
\begin{equation*}
A^{D(k''+2D-1)}\supseteq\bigoplus_{i=1}^{D}Yz_{i}.
\end{equation*}
This translates to
\begin{align*}
|A^{D(k''+2D-1)}| & \geq\prod_{i=1}^{D}|Yz_{i}|\geq|A|^{\frac{8}{7}} & & \text{in case~\eqref{th:mainzero},} \\
|A^{D(k''+2D-1)}| & \geq\prod_{i=1}^{D}|Yz_{i}|\geq\min\{|A|^{\frac{8}{7}},p^{D}\} & & \text{in case~\eqref{th:mainprime},} \\
|A^{D(k''+2D-1)}| & \geq\prod_{i=1}^{D}|Yz_{i}|\geq|A|^{1+\frac{\delta}{8\gamma}} & & \text{in case~\eqref{th:mainstuck},} \\
A^{D(k''+2D-1)} & \supseteq\bigoplus_{i=1}^{D}Kz_{i}=\mathfrak{g} & & \text{in case~\eqref{th:mainbig}.}
\end{align*}
As $k''=O(mk')=e^{O(mD^{2}\log D/\delta)}$, the theorem is proved.
\medskip

\textbf{Step 4: $A$ or $\mathrm{char}(K)$ small.} We now forgo the hypotheses that we set in Step 1. Assume first that $A$ is small, namely $|A|\leq\max\{c_{1},e^{c_{2}D}\}$ for some constants $c_{1},c_{2}$. Denote by $\varepsilon_{0},k_{0}$ the values of $\varepsilon,k$ for which the theorem is true for all large sets. By Lemma~\ref{le:linsize}, there is $h=e^{O(D)}$ for which $A^{h}$ is large: since $|A|\leq|A^{h}|$, cases \eqref{th:mainzero}, \eqref{th:mainprime}, \eqref{th:mainbig} of the theorem hold automatically for $A$ small with $\varepsilon=\varepsilon_{0}$ and $k=hk_{0}$.

Consider case~\eqref{th:mainstuck}: by hypothesis there is no $K'$ with $|A|^{\frac{1-\delta}{\dim(\mathfrak{g})}}\leq|K'|\leq|A|^{\frac{1+\delta}{\dim(\mathfrak{g})}}$. Assume that $|A^{h}|\leq|A|^{1+\frac{\delta}{3}}$ (otherwise we are done). Then
\begin{align*}
|A^{h}|^{\frac{1-\delta/2}{\dim(\mathfrak{g})}} & \geq|A|^{\frac{1-\delta/2}{\dim(\mathfrak{g})}}>|A|^{\frac{1-\delta}{\dim(\mathfrak{g})}}, \\
|A^{h}|^{\frac{1+\delta/2}{\dim(\mathfrak{g})}} & \leq|A|^{\left(1+\frac{\delta}{2}\right)\left(1+\frac{\delta}{3}\right)\frac{1}{\dim(\mathfrak{g})}}\leq|A|^{\frac{1+\delta}{\dim(\mathfrak{g})}},
\end{align*}
so there is no $K'$ with $|A^{h}|^{\frac{1-\delta/2}{\dim(\mathfrak{g})}}\leq|K'|\leq|A^{h}|^{\frac{1+\delta/2}{\dim(\mathfrak{g})}}$. Hence, case~\eqref{th:mainstuck} holds for $A$ small with $\varepsilon=\frac{\varepsilon_{0}}{2}$ and $k=hk_{0}^{2}$.

Finally, let $\mathrm{char}(K)=p$ with $0<p<3D$. By Lemma~\ref{le:linsize} then $|A^{k}|\geq\min\{(3D)^{D},|\langle A\rangle|\}\geq p^{D}$ for some $k=e^{O(D\log D)}$, and \eqref{th:mainprime} follows.

\subsection{Proof of Corollary~\ref{co:diam}}

Much like with groups (see \cite[\S 6.1]{BDH21}), once we have growth of sets inside a finite Lie algebra $\mathfrak{g}$, it is straightforward to obtain bounds on the diameter of $\mathfrak{g}$.

Let $D:=\dim(\mathfrak{g})$. Apply Theorem~\ref{th:main}\eqref{th:mainprime} to $A$, $A^{k}$, $A^{k^{2}}$... until we cover $\mathfrak{g}$, i.e.\ until $A^{k^{m}}=\mathfrak{g}$. It follows that
\begin{equation*}
m\leq\frac{\log\log|\mathfrak{g}|-\log\log|A|}{\log(1+\varepsilon)}=O(\log\log p+\log D)
\end{equation*}
and
\begin{equation*}
\mathrm{diam}(\mathfrak{g})\leq k^{m}=e^{O(D^{2}\log D(\log\log p+\log D))}=(D\log p)^{O(D^{2}\log D)},
\end{equation*}
proving the bound.

\appendix
\section{Appendix}

\subsection{Proof of Theorem~\ref{th:extremalbasis}\eqref{th:extremalbasisnofour}}\label{se:appextr}

We prove Theorem~\ref{th:extremalbasis}\eqref{th:extremalbasisnofour} here and not in the main body of the paper, since it involves a lengthier case-by-case analysis. For each $\mathfrak{g}$, we build explicitly an extremal basis and verify the property in question; the author suspects that \textit{any} $B$ and $b_{1}$ would do, as a matter of fact.

To produce $B$, one could start from the sets of extremal generators described in \cite[\S 8]{CSUW01}, and then use repeatedly the following result to build an explicit extremal basis.

\begin{lemma}\label{le:extrbra}
If $x,y$ are extremal, then $x+[x,y]+\frac{1}{2}\lambda_{y}(x)y$ is extremal.
\end{lemma}

\begin{proof}
See the proof of \cite[Lemma~2.4]{CSUW01}.
\end{proof}

Such a procedure is quite taxing to transcribe and to verify. Below, we follow more economical routes.

In all that follows, $\mathrm{char}(K)\neq 2$ is always implicitly assumed. For the ``very'' classical Lie algebras, we fix the following specific representations:
\begin{align*}
\mathfrak{sl}_{n}(K) & =\{x\in\mathrm{Mat}_{n}(K)|\mathrm{tr}(x)=0\}, \\
\mathfrak{so}_{n}(K) & =\{x\in\mathrm{Mat}_{n}(K)|x^{\top}=-x\}, \\
\mathfrak{sp}_{2n}(K) & =\left\{\left.\begin{pmatrix}x_{1} & x_{2} \\ x_{3} & x_{4}\end{pmatrix}\right|x_{i}\in\mathrm{Mat}_{n}(K),x_{1}=-x_{4}^{\top},x_{2}=x_{2}^{\top},x_{3}=x_{3}^{\top}\right\},
\end{align*}
where $x^{\top}$ denotes the transpose of $x$. Extremal elements for the algebras above will be built explicitly as matrices. For any two indices $1\leq i,j\leq N$ we denote by $e_{ij}$ the matrix in $\mathrm{Mat}_{N}(K)$ having $1$ at the $(i,j)$-th coordinate and $0$ everywhere else. The following elements are extremal in $\mathrm{Mat}_{N}(K)$ itself, thereby proving that they are extremal in any subalgebra as well:
\begin{align}
b & =e_{ij} & \text{with }\lambda_{b}(z) & =-2z_{ji}, \label{eq:extr1} \\
b & =e_{ij}+e_{ij'}+e_{i'j}+e_{i'j'} & \text{with }\lambda_{b}(z) & =-2(z_{ji}+z_{j'i}+z_{ji'}+z_{j'i'}), \label{eq:extr4} \\
b & =e_{ij}+e_{ij'}-e_{i'j}-e_{i'j'} & \text{with }\lambda_{b}(z) & =-2(z_{ji}+z_{j'i}-z_{ji'}-z_{j'i'}). \label{eq:extr22}
\end{align}
To find other extremal elements when necessary, we specialize to the various $\mathfrak{g}$.
\medskip

\textbf{Case $\mathfrak{g}=\mathfrak{sl}_{n}(K)$.} Using \eqref{eq:extr1} and \eqref{eq:extr22}, an extremal basis is
\begin{equation*}
B=\{e_{ij}\ |\ i,j\leq n,\ i\neq j\}\cup\{e_{ii}+e_{i(i+1)}-e_{(i+1)i}-e_{(i+1)(i+1)}\ |\ i<n\}.
\end{equation*}
Fixing $b_{1}=e_{12}$, we can check by direct computation that the function $\mu_{b}$ defined as $\mu_{b}(z)=[[z,b_{1}],[z,b]]$ is not identically $0$ for any $b\neq b_{1}$: for instance, if $z_{0}$ has $1-n$ at the $(n,n)$-th coordinate and $1$ everywhere else, then $\mu_{b}(z_{0})\neq 0$ for all $b$ at the same time.
\medskip

\textbf{Case $\mathfrak{g}=\mathfrak{sp}_{2n}(K)$.} Using \eqref{eq:extr1}, \eqref{eq:extr4} and \eqref{eq:extr22}, an extremal basis is
\begin{align*}
B= & \ \{e_{i(i+n)}\ |\ i\leq n\}\cup\{e_{(i+n)i}\ |\ i\leq n\} \\
 & \ \cup\{e_{i(i+n)}+e_{i(j+n)}+e_{j(i+n)}+e_{j(j+n)}\ |\ i<j\leq n\} \\
 & \ \cup\{e_{(i+n)i}+e_{(i+n)j}+e_{(j+n)i}+e_{(j+n)j}\ |\ i<j\leq n\} \\
 & \ \cup\{e_{ij}+e_{i(i+n)}-e_{(j+n)j}-e_{(j+n)(i+n)}\ |\ i,j\leq n\}.
\end{align*}
Fixing $b_{1}=e_{1(n+1)}$, it is again easy to check that $\mu_{b}$ is not identically $0$, say using $z_{0}$ having $1$ at any $(i,j)$-th coordinate with either $i\leq n$ or $j\leq n$, and $-1$ when $i,j>n$.
\medskip

\textbf{Case $\mathfrak{g}=\mathfrak{so}_{2n}(K)$.} New extremal elements emerge in this case. For any $i,j\leq n$ with $i\neq j$, the following are extremal:
\begin{align*}
b= & \ e_{ij}-e_{(j+n)(i+n)} & \text{with }\lambda_{b}(z)= & \ -2z_{ji}, \\
b= & \ e_{i(j+n)}-e_{j(i+n)} & \text{with }\lambda_{b}(z)= & \ -2z_{(j+n)i}, \\
b= & \ e_{(i+n)j}-e_{(j+n)i} & \text{with }\lambda_{b}(z)= & \ -2z_{j(i+n)}, \\
b= & \ e_{ii}-e_{(i+n)(i+n)}+e_{jj}-e_{(j+n)(j+n)} & \text{with }\lambda_{b}(z)= & \ -2(z_{ii}+z_{jj} \\
 & \ +e_{i(j+n)}-e_{j(i+n)}+e_{(i+n)j}-e_{(j+n)i} & & \ -z_{i(j+n)}-z_{(i+n)j}).
\end{align*}
Then, provided that $n\geq 3$, an extremal basis is given by
\begin{align*}
B= & \ \{e_{ij}-e_{(j+n)(i+n)}\ |\ i,j\leq n,\ i\neq j\} \\
 & \ \cup\{e_{i(j+n)}-e_{j(i+n)}\ |\ i<j\leq n\}\cup\{e_{(i+n)j}-e_{(j+n)i}\ |\ i<j\leq n\} \\
 & \ \cup\{e_{11}-e_{(n+1)(n+1)}+e_{ii}-e_{(i+n)(i+n)} \\
 & \ \ \ \ \ +e_{1(i+n)}-e_{i(n+1)}+e_{(n+1)i}-e_{(i+n)1}\ |\ 1<i\leq n\} \\
 & \ \cup\{e_{22}-e_{(n+2)(n+2)}+e_{33}-e_{(n+3)(n+3)} \\
 & \ \ \ \ \ +e_{2(n+3)}-e_{3(n+2)}+e_{(n+2)3}-e_{(n+3)2}\}.
\end{align*}
Taking $b_{1}=e_{1(n+2)}-e_{2(n+1)}$, for any $b\neq b_{1}$ we have $\mu_{b}(z_{0})\neq 0$ where $z_{0}$ is the matrix having $1$ at the $(i,j)$-th coordinates with $i,j\leq n$ or $i<n<j-i$ or $j<n<i-j$, having $0$ at the coordinates with $i=j+n$ or $j=i+n$, and having $-1$ everywhere else.
\medskip

\textbf{Case $\mathfrak{g}=\mathfrak{so}_{2n+1}(K)$.} First of all, we embed $\mathfrak{so}_{2n}(K)$ inside $\mathfrak{g}$ by sending $z$ to the matrix of $\mathfrak{g}$ with $z$ in the top left corner and $0$ everywhere else. The images of the elements of the extremal basis of $\mathfrak{so}_{2n}(K)$ found above are still extremal in $\mathfrak{g}$. Furthermore, for any $i,j\leq n$ with $i\neq j$, the following are extremal:
\begin{align*}
b= & \ 2e_{ij}-2e_{(j+n)(i+n)}+e_{j(i+n)}-e_{i(j+n)} & \text{with }\lambda_{b}(z)= & \ -2(2z_{ji}-z_{(j+n)i} \\
 & \ +2e_{i(2n+1)}-2e_{(2n+1)(i+n)} & & \ -2z_{(i+n)(2n+1)}), \\
b= & \ 2e_{ij}-2e_{(j+n)(i+n)}+e_{(j+n)i}-e_{(i+n)j} & \text{with }\lambda_{b}(z)= & \ -2(2z_{ji}+z_{i(j+n)} \\
 & \ +2e_{(2n+1)j}-2e_{(j+n)(2n+1)} & & \ +2z_{j(2n+1)}).
\end{align*}
Then, provided that $n\geq 3$, an extremal basis is given by the union of the image of the basis of $\mathfrak{so}_{2n}(K)$ with
\begin{align*}
& \{2e_{i1}-2e_{(n+1)(i+n)}+e_{1(i+n)}-e_{i(n+1)}+2e_{i(2n+1)}-2e_{(2n+1)(i+n)}\ |\ 1<i\leq n\} \\
& \cup\{2e_{1i}-2e_{(i+n)(n+1)}+e_{(i+n)1}-e_{(n+1)i}+2e_{(2n+1)i}-2e_{(i+n)(2n+1)}\ |\ 1<i\leq n\} \\
& \cup\{2e_{12}-2e_{(n+2)(n+1)}+e_{2(n+1)}-e_{1(n+2)}+2e_{1(2n+1)}-2e_{(2n+1)(n+1)}\} \\
& \cup\{2e_{21}-2e_{(n+1)(n+2)}+e_{(n+1)2}-e_{(n+2)1}+2e_{(2n+1)1}-2e_{(n+1)(2n+1)}\}.
\end{align*}
As in the previous case $\mu_{b}$ is not identically $0$, taking as $b_{1}$ the image of the element $b_{1}$ of $\mathfrak{so}_{2n}(K)$ and as $z_{0}$ the matrix having the same entries as the $z_{0}$ of $\mathfrak{so}_{2n}(K)$ (for $i,j\leq 2n$) and in addition $1$ at the $(i,2n+1)$-th coordinates, $-1$ at the $(2n+1,j)$-th coordinates, and $0$ at the $(2n+1,2n+1)$-th coordinate.
\medskip

Now we pass to the other classical Lie algebras. We make use of the explicit constructions contained in \cite[\S 22]{FH04} (for $\mathfrak{g}_{2}$) and \cite[\S 3]{HRT01} (for the others). The author verified the claims via the open-source software SageMath \cite{Sag19}, version 8.9: this passage was inescapable particularly for $\mathfrak{e}_{8}$, given its size. The programs used for the verification process can be found on GitHub \cite{Don22-s}.
\medskip

\textbf{Case $\mathfrak{g}=\mathfrak{g}_{2}(K)$.} This algebra has dimension $14$. An explicit multiplication table for $\mathfrak{g}_{2}$ is given in \cite[Table~22.1]{FH04}. Following the notation therein, an extremal basis $B$ is given by
\begin{align*}
b_{1} & =X_{2}, & b_{6} & =Y_{6}, & b_{11} & =X_{1}+Y_{3}+X_{5}-Y_{6}, \\
b_{2} & =Y_{2}, & b_{7} & =H_{2}+X_{2}-Y_{2}, & b_{12} & =Y_{1}+X_{2}-Y_{4}+Y_{6}, \\
b_{3} & =X_{5}, & b_{8} & =H_{1}+H_{2}+X_{5}-Y_{5}, & b_{13} & =Y_{1}-X_{2}+Y_{4}+Y_{6}, \\
b_{4} & =Y_{5}, &
b_{9} & =X_{1}+Y_{2}-X_{4}+X_{6}, & b_{14} & =Y_{1}+X_{3}+Y_{5}-X_{6}, \\
b_{5} & =X_{6}, & b_{10} & =X_{1}-Y_{2}+X_{4}+X_{6}, & &
\end{align*}
and the result holds for this $B$ and this $b_{1}$.
\medskip

\textbf{Case $\mathfrak{g}=\mathfrak{f}_{4}(K)$.} This algebra has dimension $52$. We use the $8$ generators of the $26$-dimensional representation defined in \cite[\S 3.5]{HRT01}. First we build a basis, not necessarily extremal. We denote the $8$ generators by $a_{1},\ldots,a_{8}$ (in the order they are presented in the paper). For brevity below, a string ``$n_{1}n_{2}\ldots n_{k}^{\alpha}$'' represents the element $\frac{1}{\alpha}[\ldots[[a_{n_{1}},a_{n_{2}}],a_{n_{3}}]\ldots,a_{n_{k}}]$, where $n_{i}\in\{1,\ldots,8\}$ and $\alpha\in\mathbb{Z}$. The other elements of the basis $a_{9},\ldots,a_{52}$ are defined in order by the strings
\begin{align*}
& 1 2^{1}, 1 5^{1}, 2 3^{1}, 2 6^{1}, 3 4^{1}, 3 7^{1}, 4 8^{1}, 5 6^{1}, 6 7^{1}, 7 8^{1}, 1 2 3^{1}, 2 3 2^{2}, 2 3 4^{1}, 5 6 7^{1}, 6 7 6^{2}, 6 7 8^{1}, \\
& 1 2 3 2^{1}, 1 2 3 4^{1}, 2 3 2 4^{2}, 5 6 7 6^{1}, 5 6 7 8^{1}, 6 7 6 8^{2}, 1 2 3 2 1^{2}, 1 2 3 2 4^{1}, 2 3 2 4 3^{2}, 5 6 7 6 5^{2}, 5 6 7 6 8^{1}, \\
& 6 7 6 8 7^{2}, 1 2 3 2 1 4^{2}, 1 2 3 2 4 3^{1}, 5 6 7 6 5 8^{2}, 5 6 7 6 8 7^{1}, 1 2 3 2 1 4 3^{2}, 1 2 3 2 4 3 2^{1}, 5 6 7 6 5 8 7^{2}, \\
& 5 6 7 6 8 7 6^{1}, 1 2 3 2 1 4 3 2^{2}, 5 6 7 6 5 8 7 6^{2}, 1 2 3 2 1 4 3 2 2^{4}, 5 6 7 6 5 8 7 6 6^{4}, 1 2 3 2 1 4 3 2 2 3^{4}, \\
& 5 6 7 6 5 8 7 6 6 7^{4}, 1 2 3 2 1 4 3 2 2 3 4^{4}, 5 6 7 6 5 8 7 6 6 7 8^{4}.
\end{align*}
From these we build an extremal basis $B$. The elements $b_{1},\ldots,b_{24}$ of this basis are the $a_{i}$ for $i$ in the set
\begin{equation*}
\{3,4,7,8,13,18,20,23,27,30,31,33,34,36,37,39,41,43,47,48,49,50,51,52\},
\end{equation*}
the elements $b_{25},\ldots,b_{50}$ are equal to $a_{i_{1}}+a_{i_{2}}-a_{i_{3}}$ for $(i_{1},i_{2},i_{3})$ in the set
\begin{align*}
\{ & (1, 30, 37), (2, 18, 27), (3, 11, 20), (3, 14, 7), (3, 16, 34), (3, 23, 6), (3, 31, 19), \\
& (3, 35, 52), (3, 51, 38), (4, 15, 8), (4, 21, 33), (4, 22, 43), (4, 28, 48), (4, 36, 17), \\
& (4, 41, 26), (4, 47, 32), (5, 27, 39), (7, 9, 31), (7, 52, 40), (8, 24, 36), (8, 25, 47), \\
& (8, 43, 29), (20, 42, 51), (23, 44, 52), (31, 51, 45), (34, 52, 46)\},
\end{align*}
and finally $b_{51}=a_{12}+a_{14}+a_{20}-a_{23}$ and $b_{52}=a_{10}+a_{12}+a_{14}+a_{31}-a_{34}$. The result holds for this $B$ and this $b_{1}$.
\medskip

\textbf{Case $\mathfrak{g}=\mathfrak{e}_{6}(K)$.} This algebra has dimension $78$. We use the $12$ generators of the $27$-dimensional representation defined in \cite[\S 3.3]{HRT01}: $6$ of them are explicitly given there (we denote them by $b_{1},\ldots,b_{6}$), and $6$ are their transposes (we denote them by $b_{\overline{1}},\ldots,b_{\overline{6}}$). The strings (all with $\alpha=1$)
\begin{align*}
& 1 3, 2 4, 3 4, 4 5, 5 6, 1 3 4, 2 4 3, 2 4 5, 3 4 5, 4 5 6, 1 3 4 2, 1 3 4 5, 2 4 3 5, 2 4 5 6, 3 4 5 6, 1 3 4 2 5, 1 3 4 5 6, \\
& 2 4 3 5 4, 2 4 3 5 6, 1 3 4 2 5 4, 1 3 4 2 5 6, 2 4 3 5 4 6, 1 3 4 2 5 4 3, 1 3 4 2 5 4 6, 2 4 3 5 4 6 5, 1 3 4 2 5 4 3 6, \\
& 1 3 4 2 5 4 6 5, 1 3 4 2 5 4 3 6 5, 1 3 4 2 5 4 3 6 5 4, 1 3 4 2 5 4 3 6 5 4 2
\end{align*}
define $30$ more elements $b_{7},\ldots,b_{36}$, and they also define $30$ elements $b_{\overline{7}},\ldots,b_{\overline{36}}$ if we use the transposes to construct them. Finally, we get $6$ elements $b'_{i}$ using Lemma~\ref{le:extrbra} on the pairs $(b_{i},b_{\overline{i}})$ for $1\leq i\leq 6$. The set $B$ of all $b_{i},b_{\overline{i}},b'_{i}$ is an extremal basis, and the result holds for this $B$ and this $b_{1}$.
\medskip

\textbf{Case $\mathfrak{g}=\mathfrak{e}_{7}(K)$.} This algebra has dimension $133$. We use the $14$ generators of the $56$-dimensional representation defined in \cite[\S 3.2]{HRT01}: $7$ of them are explicitly given there (we denote them by $b_{1},\ldots,b_{7}$), and $7$ are their transposes (we denote them by $b_{\overline{1}},\ldots,b_{\overline{7}}$). The strings used for $\mathfrak{e}_{6}$, together with the strings (all with $\alpha=1$)
\begin{align*} 
& 6 7,  5 6 7, 4 5 6 7, 2 4 5 6 7, 3 4 5 6 7, 1 3 4 5 6 7, 2 4 3 5 6 7, 1 3 4 2 5 6 7, 2 4 3 5 4 6 7, 1 3 4 2 5 4 6 7, 2 4 3 5 4 6 5 7, \\
& 1 3 4 2 5 4 3 6 7, 1 3 4 2 5 4 6 5 7, 2 4 3 5 4 6 5 7 6, 1 3 4 2 5 4 3 6 5 7, 1 3 4 2 5 4 6 5 7 6, 1 3 4 2 5 4 3 6 5 4 7, \\
& 1 3 4 2 5 4 3 6 5 7 6, 1 3 4 2 5 4 3 6 5 4 2 7, 1 3 4 2 5 4 3 6 5 4 7 6, 1 3 4 2 5 4 3 6 5 4 2 7 6, 1 3 4 2 5 4 3 6 5 4 7 6 5, \\
& 1 3 4 2 5 4 3 6 5 4 2 7 6 5, 1 3 4 2 5 4 3 6 5 4 2 7 6 5 4, 1 3 4 2 5 4 3 6 5 4 2 7 6 5 4 3, 1 3 4 2 5 4 3 6 5 4 2 7 6 5 4 3 1,
\end{align*}
define $56$ more elements $b_{8},\ldots,b_{63}$, and they also define $56$ elements $b_{\overline{8}},\ldots,b_{\overline{63}}$ if we use the transposes to construct them. Finally, we get $7$ elements $b'_{i}$ using Lemma~\ref{le:extrbra} on the pairs $(b_{i},b_{\overline{i}})$ for $1\leq i\leq 7$. The set $B$ of all $b_{i},b_{\overline{i}},b'_{i}$ is an extremal basis, and the result holds for this $B$ and this $b_{1}$.
\medskip

\textbf{Case $\mathfrak{g}=\mathfrak{e}_{8}(K)$.} This algebra has dimension $248$. SageMath can produce an explicit basis of the image of the $248$-dimensional adjoint representation of $\mathfrak{g}$, based on the construction in \cite[\S 3.1]{HRT01}: the command
\begin{center}
\verb|LieAlgebra(ZZ,cartan_type=['E',8],representation='matrix').basis()|
\end{center}
yields a family of sparse matrices, which we denote by $a_{1},\ldots,a_{248}$ (in the order they are listed by the program; we have written them down explicitly in \cite{Don22-s}). Almost all of them are already extremal, except the $8$ elements $a_{i}$ with $i$ in $I=\{121,\ldots,128\}$. We define $b_{i}:=a_{i}$ for $i\notin I$, and we note that $a_{i}=[a_{247-i},a_{129-i}]$ for all $i\in I$; then we obtain $b_{i}$ for $i\in I$ using Lemma~\ref{le:extrbra}. The newly obtained set $B$ is still a basis: in fact, for each $i\in I$ the element $b_{i}$ is in the span of $\{a_{j}|j\notin I\}\cup\{a_{i}\}$ and not in the span of $\{a_{j}|j\notin I\}$, therefore $B$ spans the same space as the original basis. The result holds for this $B$ and this $b_{1}$.

\subsection{Further generalizations}\label{se:appfurther}

We spend a few words on possible generalizations of the results, and obstacles that may lie along the way.
\smallskip

\textbf{Theorem~\ref{th:main} with smaller $k$.} We are essentially able to perform our dimensional descent because, for a $K$-vector subspace $V\subseteq\mathfrak{g}$, there is a linear homogeneous map $f$ such that $f(x_{1})=0$ and $f(x_{2})\neq 0$ for two points $x_{1},x_{2}\in V$. How quickly we find $f$ in relation to our set of generators directly impacts how large $k$ is in Theorem~\ref{th:main}.

To produce $f$, we used extremal elements and their properties. This makes for a fairly general argument that covers the more special cases $\mathfrak{e}_{6},\mathfrak{e}_{7},\mathfrak{e}_{8},\mathfrak{f}_{4},\mathfrak{g}_{2}$. However, now that Theorem~\ref{th:main} is proved, we can focus on $\mathfrak{sl}_{n},\mathfrak{so}_{n},\mathfrak{sp}_{2n}$ alone and improve the dependence of $k$ on $\dim(\mathfrak{g})$ without worrying about the other algebras. Extremal elements are in no way the only route to Proposition~\ref{pr:descent}, nor the fastest.

With a more ad hoc argument for each $\mathfrak{g}$, eliminating the logarithm factor in the exponent of $k$ is entirely possible: this improvement already beats the current best bounds for the case of groups. Any improvement on $k$ in Theorem~\ref{th:main} then leads directly to an improvement on $C$ in Corollary~\ref{co:diam}.

One may wonder how further $k$ can be improved. It would be quite significant to show that $k$ may be taken with an exponent of order $o(\sqrt{\dim(\mathfrak{g})})$: in fact, there are explicit counterexamples to product theorems of equivalent strength for groups of Lie type (see \cite[Ex.~77]{PS16}), so it would become apparent that the situation in the case of Lie algebras is fundamentally different from that of groups.
\smallskip

\textbf{Theorem~\ref{th:main} for other simple Lie algebras.} Resorting to extremal elements is quite useful for the classical Lie algebras. There might be some margin left for simple non-classical Lie algebras as well: for $K$ algebraically closed with $\mathrm{char}(K)>5$ any finite-dimensional Lie algebra has extremal elements \cite[Thm.~1]{Pre87}, and for $K$ arbitrary with $\mathrm{char}(K)\notin\{2,3\}$ the presence of extremal non-sandwich elements in a simple Lie algebra $\mathfrak{g}$ implies that $\mathfrak{g}$ is generated by extremal elements (with one exception, see \cite[Thm.~1.1]{CIR08}), and thus spanned by them too \cite[Lemma~2.4]{CSUW01}.

We would want to avoid sandwich elements however, as we did with Theorem~\ref{th:extremalbasis}\eqref{th:extremalbasisnosand}. This may be possible, but not so immediate: for $K$ algebraically closed with $\mathrm{char}(K)>5$, the only simple Lie algebras without sandwich elements are the classical ones \cite[Thm.~3]{Pre87}.

Alternatively, one might be able to work case by case on the non-classical algebras. Of course, this is possible only to the extent that a classification of simple Lie algebras is known. The work is monumental already for $K$ algebraically closed with $\mathrm{char}(K)\geq 5$, and it bears the name of \textit{Block-Wilson-Strade-Premet classification theorem}: see \cite{Str17} and the subsequent volumes.
\smallskip

\textbf{Theorem~\ref{th:main} for other algebras.} As already mentioned, our descent works by finding a linear homogeneous map $f$ with $f(x_{1})=0$ and $f(x_{2})\neq 0$. In the case of Lie algebras the first condition is automatically ensured by taking $f=[\cdot,x_{1}]$, and the second is true for classical Lie algebras (up to few manipulations) thanks to the reasoning of \S\ref{se:generic}.

For more general algebras, one may require the bracket to be an alternating form (as is the case for Mal'cev algebras) to grant at least the first condition in the same way as for Lie algebras. Algebras on the other hand of the spectrum, say those in which $[\cdot,x]$ is bijective, are likely to be out of the reach of our methods.
\smallskip

\textbf{Theorem~\ref{th:main}\eqref{th:mainstuck}-\eqref{th:mainbig} for $\mathrm{char}(K)$ small.} The condition $\mathrm{char}(K)\geq 3\dim(\mathfrak{g})$ originates from an escape step inside Corollary~\ref{co:generic}. More precisely, we know that the sets $V_{i},V'_{i}$ are not the whole $\mathfrak{g}$, but we need enough points on a line to ensure that we quickly escape from their union.

In a general sense, this sort of restriction is necessary: if $|K|$ is small and $V$ is a $K$-vector space, we can certainly arrange a number of proper subspaces of $V$ to cover the entire $V$ (see for instance \cite{Kha09}). On the other hand, one might be able to show that we can always avoid these degenerate cases in Corollary~\ref{co:generic}: this seems feasible, though convolute. It is not enough to prove that the union is proper: with careful additional choices, the arguments of Appendix~\ref{se:appextr} can already do that. One needs to prove more strongly that we can concretely escape in a bounded number of steps, regardless of the original $A$ we have at hand.

An escape argument as in \cite[\S 3]{BDH21}, i.e.\ relying on algebraic geometry, is capable of that but it is much slower: the dependence of the number of steps on $\dim(\mathfrak{g})$ becomes exponential, rather than polynomial. Alternatively, the same ideas that might quicken the search for $f$ and improve the value of $k$ (as discussed at the beginning of this subsection) would probably improve the restriction on $\mathrm{char}(K)$ as well, by the very nature of the method. In the most optimistic scenario, the only restriction left would be $\mathrm{char}(K)\neq 2$.
\smallskip

\textbf{Corollary~\ref{co:diam} for Lie algebras over $\mathbb{F}_{q}$.} As mentioned in \S\ref{se:introover}, the idea of passing from $\mathfrak{g}$ to a copy of $K$ and using sum-product estimates for $K$ comes from the proof of a diameter bound for $\mathrm{SL}_{2}(\mathbb{F}_{p})$ due to Helfgott \cite{Hel08}. Shortly after, Dinai \cite{Din11} generalized the proof to $\mathrm{SL}_{2}(\mathbb{F}_{q})$. It would be worthwhile to do the same for Lie algebras.

Here is a potential obstacle. In Dinai's proof, a set $A\subseteq K$ large enough may not grow because of the second alternative of Theorem~\ref{th:sumprodAv}, or rather of \cite[Thm.~2.55]{TV06}: there may be a subfield $K'\leq K$, an element $x\in K$, and a set $X\subseteq K$ with $A\subseteq xK'\cup X$, $|K'|\leq|A|^{1+\varepsilon}$, and $|X|\leq|A|^{\varepsilon}$. The trick to get out of the impasse is the following: since our copy of $K$ is in fact the set of possible traces, we first show that $A\subseteq K'$ cannot happen directly, up to taking a small power of $A$ (as in \cite[Cor.~2.33]{Din11}). Then, once there is an element of $A$ outside any subfield, we may use
\begin{equation*}
\mathrm{tr}(g)\mathrm{tr}(h)=\mathrm{tr}(gh)+\mathrm{tr}(gh^{-1})
\end{equation*}
to produce many such elements, and avoid that $A\subseteq xK'\cup X$ for such a small $X$.

In the second respect, Lie algebras are even more convenient: if $a\notin K'$ and $B\subseteq K'$, we have directly $(a+B)\cap K'=\emptyset$, so a simple sum $A+A$ is enough to break the possibility of $A\subseteq xK'\cup X$. However, there is no guarantee that $A\subseteq K'$ cannot happen. Unlike with groups and traces, there is no ``canonical'' copy of $K$ that we are working with: any $1$-dimensional vector space $Kv$ may become the endpoint of Theorem~\ref{th:onedim}, and there is not even a canonical way to identify $Kv$ with $K$, since it all depends on the choice of $v$.

For instance, the set
\begin{align*}
A & =\mathbb{F}_{p}a\cup\mathbb{F}_{p}b\cup\mathbb{F}_{p}c, & a & =\begin{pmatrix}1 & 0 \\ 0 & -1\end{pmatrix}, & b & =\begin{pmatrix}0 & \omega \\ 0 & 0\end{pmatrix}, & c & =\begin{pmatrix}0 & 0 \\ \omega & 0\end{pmatrix},
\end{align*}
inside $\mathfrak{g}=\mathfrak{sl}_{2}(\mathbb{F}_{q})$ generates the whole $\mathfrak{g}$ if $\omega$ is not contained in any subfield of $\mathbb{F}_{q}$. However, for every $v\in\mathfrak{g}$, the set $A\cap\mathbb{F}_{q}v$ is either $\{0\}$ or a copy of $\mathbb{F}_{p}$. We can still fill $\mathfrak{g}$ in a number of steps compatible with Corollary~\ref{co:diam}, using for example
\begin{align*}
[b,c] & =\omega^{2}a, & [[[b,c],b],c] & =2\omega^{4}a, & [[[[[b,c],b],c],b],c] & =4\omega^{6}a, & \ldots
\end{align*}
to fill $\mathbb{F}_{q}a$ quickly: but this works exactly because we keep returning to the \textit{same} copy of $\mathbb{F}_{q}$. It is not at all obvious that this phenomenon occurs in general: while $\mathfrak{sl}_{2}$-pairs are fairly common inside semisimple Lie algebras (by the Jacobson-Morozov theorem, see for instance \cite[\S III.11, Thm.~17]{Jac79} and \cite[Prop.~2.1]{CIR08}), nothing guarantees that one of them is contained in a small power of $A$. On the other hand, there is also no reason why working with $\mathfrak{sl}_{2}$-pairs would be the \textit{only} way to fill $\mathfrak{g}$. It would be quite interesting to investigate either issue.
\smallskip

\textbf{Sumset theory for Lie algebras.} As mentioned in \S\ref{se:prelfields}, sum-product estimates can be expressed through sets of different shape, and then translated into the canonical estimate for $\max\{|A+A|,|AA|\}$. Similarly, in the case of groups, from a bound of the form $|A^{k}|\geq|A|^{1+\varepsilon}$ it is easy to obtain $|A^{3}|\geq|A|^{1+\varepsilon'}$ as well.

Both of the above are consequences of elementary results in sumset theory, mainly developed by Ruzsa. The triangle inequality
\begin{equation}\label{eq:ruzsatriangle}
|AC^{-1}|\cdot|B|\leq|AB^{-1}|\cdot|BC^{-1}|,
\end{equation}
valid for any $A,B,C$ finite inside an arbitrary group (see \cite[Lemma~4.2]{Ruz96} or \cite[Lemma~2.6]{TV06} for the abelian case, and \cite[Lemma~2.1]{Hel08} in general), yields countless such manipulations. There is no sumset theory in our paper: our Lemma~\ref{le:linsize} falls short of reaching even the trivial bound $|A^{k}|\geq k$ valid for groups.

It would be interesting to develop a sumset theory for Lie algebras, or non-associative algebras in general: the main results of this paper would become more flexible, allowing for wider ranges of exponents and other manipulations. There is hope: as shown in \cite[Lemma~4.2]{GHR15}, \eqref{eq:ruzsatriangle} can already be extended to certain structures weaker than groups, mutatis mutandis.

\section*{Acknowledgements}

The author thanks his supervisor A.~Shalev for suggesting the problem of growth in non-associative algebras and for his patient reading and advice, and I.~D.~Shkredov for helpful remarks on sum-product theorems. The author is grateful to the community of MathOverflow for attracting his attention to the topic of extremal elements (via question 416738).

The author was supported by the Israel Science Foundation Grants No. 686/17 and 700/21 of A.~Shalev, and the Emily Erskine Endowment Fund; he has been a postdoc at the Hebrew University of Jerusalem under A.~Shalev in 2020/21 and 2021/22.

\bibliography{Bibliography}
\bibliographystyle{alpha}

\end{document}